\documentclass[journal,transmag]{IEEEtran}

\usepackage{cite}
\usepackage{amsmath,amssymb,amsfonts}
\usepackage{algorithmic}
\usepackage{graphicx}
\usepackage{algorithm,algorithmic}
\usepackage{hyperref}
\hypersetup{hidelinks=true}
\usepackage{textcomp}

\usepackage{amsmath, amssymb, amsfonts, amsthm}  
\usepackage{cases}
\usepackage{hyperref}
\usepackage{graphicx}    
\usepackage{mathtools}

\usepackage{bm}
\usepackage{bbm}
\usepackage{dsfont}
\usepackage{xcolor}
\usepackage{placeins}
\usepackage{ifpdf}
\usepackage{soul}
\usepackage{amsmath}

\usepackage{comment}

\usepackage{tikz}
\usetikzlibrary{arrows.meta, positioning}

\usepackage{enumitem}  
\setlist[description]{leftmargin=\parindent}
\usepackage[mathscr]{eucal}
\usepackage{subfigure}

\newtheorem{theorem}{Theorem}
\newtheorem{definition}{Definition}
\newtheorem{proposition}{Proposition}
\newtheorem{lemma}{Lemma}
\newtheorem{corollary}{Corollary}
\newtheorem{assumption}{Assumption}
\newtheorem{problem}{Problem}
\newtheorem*{problem2A}{Problem 2A}
\newtheorem*{problem2B}{Problem 2B}
\newtheorem*{problem2C}{Problem 2C}

\theoremstyle{remark}

\newtheorem{remark}{Remark}
\newtheorem{example}{Example}

\def\qed{\hfill \vrule height 5pt width 5pt depth 0pt \medskip}

\def\diag{{\rm diag}}

\def\im{{\rm Im}}

\newcommand{\Am}{\mathcal{A}}
\newcommand{\Bm}{\mathcal{B}}
\newcommand{\Cm}{\mathcal{C}}
\newcommand{\Dm}{\mathcal{D}}
\newcommand{\Em}{\mathcal{E}}
\newcommand{\Gm}{\mathcal{G}}
\newcommand{\Hm}{\mathcal{H}}
\newcommand{\Jm}{\mathcal{J}}
\newcommand{\Km}{\mathcal{K}}

\newcommand{\Pm}{\mathcal{P}}

\newcommand{\Sm}{\mathcal{S}}
\newcommand{\Tm}{\mathcal{T}}
\newcommand{\Vm}{\mathcal{V}}

\newcommand{\Wm}{\mathcal{W}}
\newcommand{\Fm}{\mathcal{F}}
\newcommand{\Xm}{\mathcal{X}}
\newcommand{\Ym}{\mathcal{Y}}
\newcommand{\Zm}{\mathcal{Z}}

\newcommand{\Vmin}{\Bm}
\newcommand{\Vmout}{\Cm}

\newcommand{\beq}{\begin{equation}}
\newcommand{\eeq}{\end{equation}}
\newcommand{\beqa}{\begin{eqnarray}}
\newcommand{\eeqa}{\end{eqnarray}}
\newcommand{\beqan}{\begin{eqnarray*}}
\newcommand{\eeqan}{\end{eqnarray*}}

\newcommand{\bite}{\begin{itemize}}
\newcommand{\eite}{\end{itemize}}
\newcommand{\benu}{\begin{enumerate}}
\newcommand{\eenu}{\end{enumerate}}

\newcommand\rednote[1]{\textcolor{red}{#1}}

\definecolor{darkpastelgreen}{rgb}{0.01, 0.75, 0.24}

\definecolor{fuchsia}{HTML}{FF00FF}
\definecolor{softblue}{HTML}{007FFF}
\definecolor{softblue}{HTML}{007FFF}
\definecolor{darkgreen}{HTML}{009900}

\begin{document}
\title{Geometric Control Theory Over Networks: Minimal Node Cardinality Disturbance Decoupling Problems}
\author{Luca Claude Gino Lebon, \IEEEmembership{Graduate Student Member, IEEE}, and Claudio Altafini, \IEEEmembership{Senior Member, IEEE}
\thanks{Work supported in part by the Swedish Research Council (grant n. 2024-04772 to C.A.) and by the ELLIIT framework program at Link\"oping University. }
\thanks{L. C. G. Lebon and C. Altafini are with the Department of Electrical Engineering, Division of Automatic Control,
        Link\"oping University, SE-581 83 Link\"oping, Sweden
        {\tt\small luca.lebon@liu.se,claudio.altafini@liu.se}}%
}

\maketitle

\begin{abstract}
In this paper we show how to formulate and solve disturbance decoupling problems over networks while choosing a minimal number of input and output nodes. 
Feedback laws that isolate and eliminate the impact of disturbance nodes on specific target nodes to be protected are provided using state, output, and dynamical feedback. 
For that, we leverage the fact that when reformulated in terms of sets of nodes rather than subspaces, the controlled and conditioned invariance properties admit a simple graphical interpretation. 
For state and dynamical feedback, the minimal input and output cardinality solutions can be computed exactly in polynomial time, via min-cut/max-flow algorithms. 
\end{abstract}
\section{Introduction}
\label{sec:introduction}
The Disturbance Decoupling Problem (DDP) is a basic problem in control theory, and consists in rendering a set of variables (which we refer to as targets) insensitive to disturbances entering into the model of the system. 
It is known since the late Sixties that for linear systems the DDP can be solved in terms of invariant subspaces \cite{basile1969controlled,wonham1970decoupling}, and in fact it has been one of the main motivating problems behind the development of geometric control theory \cite{basile1992controlled,wonham2013linear,trentelman2012control}. 
It relies on notions such as controlled and conditioned invariance, and on the feedback synthesis that follows from them.
These concepts are normally formulated for subspaces and, while well-understood algebraically and geometrically, they remain computationally challenging because of the numerical ill-conditioning associated with operations on subspaces, such as invariance, intersection, sum etc.

In this paper we show that when applied to (linear) networked control systems in which disturbances, targets, measurements and controls all act on individual nodes (rather than on linear combination of nodes), many of the concepts of geometric control can be simplified drastically.  
For instance, subspaces can be replaced by sets of nodes, and invariance of a subspace by lack of outgoing edges from a set of nodes. The controlled invariance property of a set of nodes corresponds to having edges that stay in the set or, when they exit it, land on the input nodes. 
Similarly, conditioned invariance corresponds to a set of nodes whose edges can leave the set only if they start on the outputs. 

These interpretations of controlled/conditioned invariance, which were first presented in our knowledge in \cite{conte2019invariance}, are here developed in full and exploited.
Their simplicity and graphical meaning, in fact, pave the way for a systematic application of all tools from geometric control theory to networked systems. 
In particular, for the DDP, we show that its simplest version, DDP by State Feedback (DDPSF), corresponds to checking a set containment condition involving the maximal controlled invariant set of nodes. Also more advanced forms of DDP, like DDP by Output Feedback (DDPOF) and DDP by Dynamical Feedback (DDPDF) can be reformulated along similar lines. 
All three correspond graphically to being able to block (or to compensate exactly for) all paths from disturbances to targets on the graph of the system through a feedback design.

The simplicity of the setting allows also to solve a more complex and novel problem, namely determine what is the most parsimonious selection of input (and output) nodes that can solve a DDP. We formulate this problem as a minimal cardinality problem with respect to the number of inputs in the DDPSF, and with respect to the number of inputs and outputs in the DDPOF and DDPDF. 
All can be mapped to graphical conditions, which can be solved optimally for DDPSF and DDPDF in polynomial time using variants of min-cut/max-flow.

In all three cases, we also provide feedback laws that achieve the decoupling. 
For the DDPSF, the feedback consists in exactly canceling the arcs immediately upstream of the input nodes. 
For the DDPOF, feasibility corresponds to output and input laying on the inner and outer boundary of a set of nodes which must be simultaneously controlled and conditioned invariant, and the output feedback corresponds to canceling all edges that cross the boundary. 
For the DDPDF, an observer-based feedback is instead used, in which the observer corresponds to the nodes lying between the maximal controlled invariant set and the minimal conditioned invariant set.

While the classical geometric conditions written in terms of the maximal controlled invariant subspace and the minimal conditioned invariant subspace are necessary and sufficient for the solvability of the DDP (by state, output and dynamical feedback), the conditions written in terms of node sets are instead only sufficient. Nevertheless, their intuitive and simple characterization leads to remarkably practical solutions.

\noindent {\it Literature review.} The classical literature on DDP is summarized in well-known geometric control books such as \cite{basile1992controlled,wonham2013linear,trentelman2012control}. When it comes to networked systems, we notice that there exists another graphical approach to the DDP, based on the notion of structured systems \cite{commault1997geometric,dion1993feedback,dion1994simultaneous}, i.e., on seeking conditions valid for almost all values of the system parameters. 
In this context, it is worth mentioning e.g. \cite{commault1991disturbance,van1991structure}, which provide necessary and sufficient conditions for the DDP from state feedback in terms of the structure at infinity of the structured system, and \cite{van1996disturbance}, where these conditions are extended to the DDP with measurement feedback. 
While the structural approach is related to ours, a one-to-one mapping is impossible because this approach uses different conventions in its problem formulation (e.g., nodes that carry inputs are separated from ``state'' nodes, and from output nodes), and the input, output and disturbance matrices are not elementary columns/rows as in our approach, meaning that an explicit association between invariant sets of nodes and invariant subspaces is absent. 
In any case, none of the classical papers deals with the minimal cardinality control and output assignment problems which are at the core of our work.

Apart from the aforementioned structured systems and \cite{conte2019invariance}, some other sporadic efforts have been made to translate the invariance conditions into network terms. For instance, the paper \cite{monshizadeh2015disturbance} establishes sufficient conditions for disturbance decoupling in diffusive-type multi-agent systems by means of graph partitions.

A broad literature, starting with \cite{willems1981disturbance}, exists also on the joint problem of DDP and stability, or more generally on pole placement in DDP.  
In this perspective, recent studies have formulated the DDP as a constraint within an optimization framework \cite{sarsilmaz2024revisiting}.
A joint minimal input cardinality DDP and pole placement problem is investigated for instance in  \cite{clark2015input} for general linear systems, using very different tools such as matroids and submodular optimization.
On a network, the extra requirements associated with pole allocation may hamper the graphical interpretation and analysis. We therefore decided to disregard them, since, as it is well known, stability can always be recovered by adding sufficiently many (negative) self-loops, until diagonal dominance is achieved in the state update matrix \cite{horn2012matrix}. Self-loops do not interfere with our DDP formulation. 

Finally, while in this paper we are not focusing on any specific application, it is worth mentioning that there exists a wide range of networked systems in which DDPs are of relevance. Apart from classical problems of protecting strategic nodes in real-world networks such as power grids, traffic systems, IT infrastructures, examples from the recent literature include \cite{morbidi2007leader}, where leader-follower formation control is framed as a DDP, or \cite{van2000noncooperative}, where DDP is formulated within a dynamic game framework. Another context in which our tools could be applied is fault detection and isolation, see for instance \cite{jia2020fault}, or \cite{lunze2006control}, where control reconfiguration after actuator failures is shown to be equivalent to a DDP. The related problem of attack detection and identification can also be rephrased along similar lines, see, e.g., \cite{pasqualetti2013attack}.

The rest of the paper is organized as follows. After a recap of geometric control theory in Section~\ref{sec:geom-contr}, the DDP problem over networks is solved in Section~\ref{sec:standard-DDP}, while in Section~\ref{sec:min-node-DDP} we formulate and solve the minimal-node DDP. The synthesis of feedback laws is in Section~\ref{sec:feedback}.
The proofs of all results are given in the Appendices.

A preliminary version of this paper \cite{LebonCDC25} deals only with the DDPSF case and lacks the proofs of the main results. The material on DDPOF and DDPDF is presented here for the first time.

MATLAB functions implementing the results presented in this article are available in the GitLab repository \cite{mycode2026}.
\section{Elements of geometric control theory over networks}
\label{sec:geom-contr}

In this section, standard material from geometric control books such as \cite{basile1992controlled,trentelman2012control,wonham2013linear} is reformulated for networked systems.

\subsection{Notation}
We consider a digraph $ \Gm = ( \Vm, \, \Em , A )$, where $ \Vm =\{ v_1, \, \ldots, v_n\}$ is a set of $n$ nodes and $\Em $ a set of $q$ edges over $ \Vm$, with the convention that $ (v_i , \, v_j) \in \Em $ means $ v_i \to v_j $, i.e., $ v_i $ is the tail and $ v_j $ is the head of the edge. 
$ A\in\mathbb{R}^{n\times n}$ is the adjacency matrix associated with $ \Gm$: $ A_{ji} $ is the weight of the edge $ (v_i , \, v_j) \in \Em $. 

Denote $ \Am $ the class of adjacency matrices of $ \Gm$ having ``support" in $ \Em$: $\Am = \{ A \in\mathbb{R}^{n\times n} \; \text{ s.t.} \; A_{ji}\neq 0\;  \text{iff} \; (v_i , \, v_j) \in \Em\; \text{and} \; A_{ji}= 0\;  \text{iff} \; (v_i , \, v_j) \notin \Em \}$.
A special matrix in $ \Am $ is the indicator matrix $\mathbb{A}$, i.e., the unweighted version of $A$ of entries $\mathbb{A}_{ij}=\mathbb{I}(A_{ij} \neq 0)$, with $\mathbb{I}(\cdot)$ denoting the indicator function, that is, 1 if the condition is true and 0 otherwise.

Given a digraph $ \Gm$, a subset $ \Zm \subset \Vm$ is called a terminal subset of nodes if there is no outgoing edge from $ \Zm$, i.e., $\forall\,(v_i , \, v_j) \in \Em $ if $ v_i \in \Zm $ then also $ v_j \in \Zm$.
A directed path $ \Pm$ from node $v_i$ to node $v_j$ is a sequence of distinct nodes $v_i, v_{i+1}, \dots , v_{j-1}, v_j$ s.t. $(v_k,v_{k+1})\in \mathcal{E}$ $\forall\,k\in\{i,i+1, \dots, j-1\}$: $\Pm = \{ v_i,  \dots , v_j \} $ (or $ \Pm^{v_i, v_j} $ when we need to specify the extremities of the path). 
The number of edges in the path is called the length of the path. Given a set of nodes $ \Zm \subseteq \Vm$, with a slight abuse of notation, we can use $ \Zm $ to express subspaces of $ \mathbb{R}^{n}$: $\Zm \sim \left\{  x \in \mathbb{R}^n  \; \text{s.t.} \; x_i \neq 0\;  \forall \; i \in \Zm  \; \text{and} \; x_i= 0\; \forall \; i  \notin \Zm \right\}$.
In the following, we explicitly fix a basis of canonical vectors for a vector subspace associated with a set of nodes.
Let $ e_i $ be an elementary column vector (i.e., the vector having $1$ on th $i$-th slot and $0$ elsewhere).
\begin{definition}[\textbf{Subspace associated with a set of nodes}]\label{def:sub}
Given a set of nodes $\Zm\subseteq\Vm$, the vector subspace of $\mathbb{R}^n$ associated with $\Zm$ is $\textnormal{span}\{e_i\;\text{s.t.}\;v_i\in\Zm\}$.
\end{definition}
Note that in general a vector subspace can have non-canonical generators, whenever its basis vectors necessarily involve linear combinations of canonical  vectors. Throughout the paper, we identify $ \Zm $ and the associated vector subspace, and sometimes write $ x \in \Zm $.
\begin{remark}
Since a vector subspace contains the $0$ element, the identification ``sets of nodes'' $ \sim $ ``vector subspaces'' holds almost always, except for a zero measure set. This is the same principle behind the structural linear systems approach \cite{commault1997geometric}. The only case in which this fact has to be made explicit is when we have equalities involving some matrix $A\in \Am$, such as $ A x = y $, where $ x \in \Xm$, $ y\in \Ym$ with $ \Xm , \, \Ym $ sets of nodes / vector subspaces.
In this case, to avoid complications, we always replace $ A$ with the indicator matrix $ \mathbb{A}$, and consider the expression $ \mathbb{A} \Xm=\Ym$. 
This is not needed when we have inequalities such as $ A \Xm \subseteq \Ym$.
\end{remark}

For any $ \Zm \subset \Vm$, the notation $ \Zm^\perp$ is used for both the complement set $ \Vm \smallsetminus \Zm$ and the subspace orthogonal to (the vector space) $ \Zm $ in $ \mathbb{R}^n$.



\begin{remark}
For each subspace $ \Zm$, $ \exists $ a full column rank matrix $ Z$ s.t. $ \im Z = \Zm $ and $ \ker Z =\{0\}$.  
Such matrix is called a basis matrix of $\Zm$. When $ \Zm $ is a set of nodes, then $ Z$ is a collection of elementary column vectors (one column $ e_i$ for each node $ v_i \in \Zm$). In this case, the sum of vector subspaces and their union can be identified: $ \Zm_1 + \Zm_2 \sim \Zm_1 \cup \Zm_2 $. Moreover, $Z^\top$ is full-row rank and $\ker Z^\top=(\im Z)^\perp$. If $ \Sm \subset \Vm$ denotes another set of nodes, then $ \Sm \cap \ker Z^\top$ corresponds to the component of $\Sm$ orthogonal to $\im Z$. In this setting, set subtraction and orthogonal complement coincide because all subspaces involved are spanned by disjoint subsets of orthonormal basis vectors. We therefore write $\Sm \cap \ker Z^\top\sim\Sm \smallsetminus \im Z$.
\end{remark}
On $ \Gm = ( \Vm, \, \Em , A )$, for any subset of nodes $ \Wm\subseteq\Vm$ we can define the in- and out-boundary as follows.
\begin{definition}\label{def:boundary}
The in-boundary of $\Wm\subseteq\Vm$ in the digraph $ \Gm = ( \Vm, \, \Em , A )$ is the set of nodes in $\Wm$ that have at least one outgoing edge leading to a node in $\Wm^\perp$, i.e., $\partial_-(\Wm,A):=\{v_i\in\Wm\;\text{s.t.}\;\exists\,(v_i,v_j)\in\Em,\, v_j\in\Wm^\perp\}$. Similarly, the out-boundary of $\Wm$ is the set of nodes in $\Wm^\perp$ that have at least one incoming edge from $\Wm$, i.e., $\partial_+(\Wm,A):=\{v_j\in\Wm^\perp\;\text{s.t.}\;\exists\,(v_i,v_j)\in\Em,\, v_i\in\Wm\}$.
\end{definition}
From the previous definition, it is obvious that $\partial_\pm(\Wm,A)=\partial_\mp(\Wm^\perp,A^\top)$.
Notice that in general $ \partial_+(\Wm,A) \subseteq A \partial_-(\Wm,A)$ and the inclusion can be strict. 

\subsection{Linear systems over networks}

In $ \Gm $, we assume that two subsets of nodes play the role of inputs and outputs, denoted, respectively, $ \Vmin \subseteq \Vm$ and $ \Vmout \subseteq \Vm$.
We can form a linear time-invariant control system having a state update matrix equal to $A$, input matrix $B\in\mathbb{R}^{n\times m}$ composed of $m$ elementary column vectors corresponding to $ \Vmin $, and output matrix $C\in\mathbb{R}^{p\times n}$ having $p$ elementary rows corresponding to $ \Vmout$:
\beq
\begin{split}
\dot x & =  Ax + B u, \\ 
y & = C x ,
\end{split}
\label{eq:lin-syst1}
\eeq
where we have associated one state variable $ x_i \in \mathbb{R} $ of the vector $ x = [ x_1 \ldots x_n]^\top $ to each node $ v_i $ of the graph.

In some cases, we need to highlight the topological structure underlying the system~\eqref{eq:lin-syst1}.
For that, notice that, by construction, 
$ \Vmin= \im B$ and  $ \Vmout = \im C^\top=(\ker C)^\perp$.

\subsection{Invariance, controlled invariance and conditioned invariance for node sets}
In this section we formulate the concepts of invariance for sets of nodes, rather than for subspaces, as it is standard in the literature (see, e.g., Chapters 4--6 of \cite{trentelman2012control}).

\begin{definition}
\label{def:A-invar1}
A subset $ \Zm \subset \Vm$ is an invariant (or $A$-invariant) set of nodes if $  A \Zm \subseteq \Zm $.
\end{definition}

Definition~\ref{def:A-invar1}, a standard concept also for subspaces, corresponds to $Ax\in\Zm$, $\forall\,x\in\Zm$. For network systems like~\eqref{eq:lin-syst1} we can have also an equivalent topological characterization, valid for the entire class $ \Am$. 

\begin{proposition}
\label{prop:invA}
Consider the system~\eqref{eq:lin-syst1}. A subset $ \Zm \subset \Vm $ is an $A$-invariant set of nodes if and only if any of the following equivalent conditions is met:
  \benu[label=(\alph*)]
  \item $ \Zm$ is a terminal subset of nodes;
\item If $ Z $ is the basis matrix of $ \Zm $ (i.e., $ \im Z = \Zm $ and $Z$ is full column rank), $ \exists$ a matrix $X$ s.t. $AZ = ZX$. 
  \eenu
\end{proposition}
\noindent The proof of this proposition (as well as all other proofs) is given in the Appendix.

For a system like \eqref{eq:lin-syst1}, the concepts of controlled invariance (also called $ (A,B)$-invariance) and conditioned invariance (or $ (C,A)$-invariance) serve to expand the notion of invariance.

\begin{definition}
\label{def:contr-invar}
A subset $ \Zm \subset \Vm$ is a controlled invariant (or $(A,B)$-invariant) set of nodes for~\eqref{eq:lin-syst1} if $ A \Zm \subseteq \Zm + \Vmin $ (with $\Vmin=\im B$ and $B$ with elementary columns).
\end{definition}

The following proposition provides three equivalent characterizations of $ (A,B)$-invariance, one topological and the other two valid for a given matrix $A$. 
\begin{proposition}
\label{prop:AB-invar-equiv-char}
$ \Zm \subset \Vm$ is an $ (A,B)$-invariant set of nodes for~\eqref{eq:lin-syst1} if and only if any of the following equivalent conditions is met:
\benu[label=(\alph*)]
\item $\forall\,(v_i , \, v_j) \in \Em $, $ v_i \in \Zm \; \Longrightarrow \; v_j \in \Zm \cup \Vmin$;\label{item:a_prop:AB-invar-equiv-char}
\item If $ Z $ is the basis matrix of $ \Zm $, $ \exists $ matrices $X$, $U$ s.t. $AZ = Z X +BU$;\label{item:b_prop:AB-invar-equiv-char}
\item $\exists$ a matrix $F\in\mathbb{R}^{m\times n}$, called a ``friend'' of $\Zm$, s.t. $\Zm$ is $(A-BF)$-invariant, i.e., $(A-BF)\Zm\subseteq\Zm$.\label{item:c_prop:AB-invar-equiv-char}

\eenu
\end{proposition}

\begin{definition}
A subset $ \Sm \subset \Vm$ is a conditioned invariant (or $ (C,A)$-invariant) set of nodes for~\eqref{eq:lin-syst1} if $ A ( \Sm \smallsetminus\Vmout)\subseteq \Sm$ (with $\Vmout=(\ker C)^\perp$ and $C$ with elementary rows).
\end{definition}
Similarly to above, equivalent characterizations of $ (C,A)$-invariance can be formulated.
\begin{proposition}\label{prop:CA-invar-equiv-char}
$\Sm $ is a $ (C,A)$-invariant set of nodes if and only if any of the following equivalent conditions is met:
\benu[label=(\alph*)]
\item $\forall\,(v_i , \, v_j) \in \Em $, $ v_i \in \Sm$, $ v_i \notin \Vmout$ $ \Longrightarrow \; v_j \in \Sm$;\label{item:a_prop:CA-invar-equiv-char}
\item $\exists$ a matrix $H\in\mathbb{R}^{n\times p}$, called a ``friend'' of $\Sm$, s.t. $\Sm$ is $(A-HC)$-invariant, i.e., $(A-HC)\Sm\subseteq\Sm$.\label{item:b_prop:CA-invar-equiv-char}
\eenu
\end{proposition}


\begin{definition}
A pair of subsets $(\Sm,\,\Zm)$ of $\Vm$ is called a $(C,A,B)$-pair of node sets if
\bite
\item $\Sm\subseteq\Zm$;
\item $\Sm$ is $(C,A)$-invariant;
\item $\Zm$ is $(A,B)$-invariant.
\eite
\end{definition}

\subsection{Maximal controlled invariance, and minimal conditioned invariance for node sets}
\label{sec:maximalcontrolled}
The algorithms in Propositions~\ref{prop:Z} and~\ref{prop:S} can be found in~\cite{conte2019invariance} and return sets of nodes (rather than subspaces). They are the network counterpart of well-known algorithms in the literature, see Propositions~\ref{prop:contr_max_semi} and \ref{prop:cond_min_semi} in Appendix~\ref{app:subsp-max-min} \cite{basile1969controlled,wonham1970decoupling}.
More details on the differences with the classical case are given in Section~\ref{sec:rel}.

\begin{proposition}[\textbf{Maximal Controlled Invariant Node Set} $\mathbf{\Zm^\circ}$]\label{prop:Z}
Given a set $\Zm\subseteq\Vm$ and the recursion

\bite
\item$ \Zm_0 = \Zm$;
\item$ \Zm_{k+1} = \Zm_k \smallsetminus \{ v_i \in \Zm_k\,\,\emph{s.t.}\,\exists \; \text{at least one}\,\,(v_i, \, v_j ) \in \Em\,\, \text{with}\,\,v_j \notin \Zm_k \cup \Vmin \}$,
\eite
$\exists$ an index $r\leq n-1$ s.t. $\Zm_{r+p}=\Zm_{r}$, $\forall\,p\geq1$. The set $\Zm^\circ\coloneqq\Zm_r$ is the maximal control invariant subset of $\Zm$.
\end{proposition}

By construction $\Zm_{k+1}\subseteq\Zm_{k}$, therefore $\Zm^\circ\subseteq\Zm_0$. 
Notice that the maximal control invariant node set is a function of $ \Zm$, $ \Vmin$ and of the topology of $ \Gm$: $ \Zm^\circ = \Zm^\circ(\Vmin, \Zm, A)$. 
$ \Zm^\circ $ can possibly be an empty set.

\begin{proposition}[\textbf{Minimal Conditioned Invariant Node Set} $\mathbf{\Sm^\circ}$]\label{prop:S}
Given a set $\Sm\subseteq\Vm$ and the recursion
\bite
\item $ \Sm_0 = \Sm$;
\item $ \Sm_{k+1}= \Sm_k \cup \{ v_j \notin \Sm_k\,\,\emph{s.t.}\,\,\exists \; \text{at least one} \,\,(v_i , \, v_j ) \in \Em\,\,\text{with}\,\,v_i \in \Sm_k \smallsetminus \Vmout \} $,
\eite
$\exists$ an index $r\leq n-1$ s.t. $\Sm_{r+p}=\Sm_{r}$, $\forall\,p\geq1$. The set $\Sm^\circ\coloneqq\Sm_r$ is the minimal conditioned invariant superset of $\Sm$.
\end{proposition}

By construction $\Sm_{k}\subseteq\Sm_{k+1}$, therefore $\Sm_0\subseteq\Sm^\circ$.
Similarly to above, we have $ \Sm^\circ = \Sm^\circ(\Vmout, \Sm, A)$, and $ \Sm^\circ $ can possibly be equal to $ \Vm$.

\subsection{Duality}
It is well known that the notions of controlled and conditioned invariance are dual notions. 
Here we repeat these arguments for sets of nodes, rather than subspaces.

\begin{lemma}[\textbf{Property 3.1.2 of \cite{basile1992controlled}}]\label{lem:dual}
Given a linear map identified by the matrix $A:\mathbb{R}^n\to\mathbb{R}^n$, and any two subsets $ \Xm, \Ym \subset \Vm$, the following equivalence holds:
\begin{equation*}
A\mathcal{X}\subseteq\mathcal{Y}\iff A^\top\mathcal{Y}^\perp\subseteq\mathcal{X}^\perp.
\end{equation*}
\end{lemma}
\begin{proposition}[\textbf{Property 4.1.3 of \cite{basile1992controlled}}]\label{prop:dual}
The orthogonal complement of a $(C,A)$- (resp. $(A,B)$-) invariant set is a $(A^\top,C^\top)$- (resp. $(B^\top,A^\top)$-) invariant set.
\end{proposition}
When applied to maximal controlled / minimal conditioned invariant node sets, we get the following duality result. 
\begin{proposition}\label{prop:dual_net}
Given the node sets $\Zm^\circ=\Zm^\circ(\Vmin,\,\Vm\smallsetminus\Tm, A)$ and $\Sm^\circ=\Sm^\circ(\Vmout,\,\Dm, A)$, where $ \Dm, \, \Tm \subset \Vm$ are node sets s.t. $ \Dm \cap \Tm =\emptyset$, we have
\begin{align*}
\Zm^{\circ\perp}&=\Sm^\circ(\Vmin,\,\Tm, A^\top);\\
\Sm^{\circ\perp}&=\Zm^\circ(\Vmout,\,\Vm\smallsetminus\Dm,A^\top),
\end{align*}
where $\Sm^\circ(\Vmin,\,\Tm, A^\top)$ is the minimal $(B^\top,A^\top)$-invariant node set containing $\Tm$ when the output set is $ \Vmin$ and $\Zm^\circ(\Vmout,\,\Vm\smallsetminus\Dm, A^\top)$ is the maximal $(A^\top,C^\top)$-invariant node set contained in $\Vm\smallsetminus\Dm$ when the input set is $ \Vmout$.
\end{proposition}

\subsection{Node set vs subspace characterization}\label{sec:rel}

Rather than the {\em sets of nodes} $\Zm^\circ$ and $\Sm^\circ$ computed in Propositions~\ref{prop:Z} and \ref{prop:S}, the classical literature on geometric control makes use of {\em subspaces}, the maximal controlled invariant subspace $\Zm^\ast$, and the minimal conditioned invariant subspace $\Sm^\ast$. 
These are computed through the recursions of Propositions~\ref{prop:contr_max_semi} and \ref{prop:cond_min_semi} in Appendix~\ref{app:subsp-max-min}.
We now discuss the difference between $\Zm^\ast$, $\Sm^\ast$ and $\Zm^\circ$, $\Sm^\circ$, assuming the recursions start from the same initial conditions.
In particular, the following Theorem states that the subspaces $\Zm^\ast$ and $\Sm^\ast$ need not coincide with the node sets $\Zm^\circ$ and $\Sm^\circ$.
\begin{theorem}\label{thm:relat}
As vector subspaces, $\Sm^\circ\supseteq\Sm^\ast$ and $\Zm^\circ\subseteq\Zm^\ast$. In particular:
\benu[label=(\alph*)]
    \item If $x\in\Sm^\ast$ is s.t. generically $x_i\neq0\implies v_i\in\Sm^\circ$;\label{item:S_1}
    \item If $v_i\in\Zm^\circ\implies x\in\Zm^\ast$, generically $x_i\neq 0$.\label{item:Z_1}
\eenu
\end{theorem}

\begin{remark}\label{rem:relat}
Theorem~\ref{thm:relat} shows that there are no nonzero coordinates in the generators of $\Sm^\ast$ that are not nodes of $\Sm^\circ$, and there are no nodes in $\Zm^\circ$ that are not nonzero coordinates in $\Zm^\ast$. The reverse implications of items~$\ref{item:S_1}$ and $\ref{item:Z_1}$ of Theorem~\ref{thm:relat} are not true in general (see Example~\ref{ex:ex0}). Since by construction $ \Zm^\circ \subseteq \Zm^\ast $ and $ \Sm^\circ \supseteq \Sm^\ast $, all the DDP conditions that we derive below are also valid in the classical sense. Note that the concepts of in- and out-boundary of a set of nodes, as in Definition~\ref{def:boundary}, naturally apply to $\Zm^\circ$ and $\Sm^\circ$, respectively, but not, in general, to $\Zm^\ast$ and $\Sm^\ast$.
\end{remark}
\begin{remark}\label{rem:struct}
For completeness, we mention that in \cite{commault1997geometric} the maximal {\em fixed structured} controlled invariant and the minimal {\em fixed structured} conditioned invariant subspaces are defined, here denoted by $\Zm^g$ and $\Sm_g$, respectively. Those sets differ from our $\Zm^\circ$ and $\Sm^\circ$. ``Fixed'' for structured systems corresponds to elementary basis vectors which remain in a set for almost all $ A \in \mathcal{A}$. Those fixed subspaces are generated by canonical basis vectors and thus are in one-to-one correspondence with sets of nodes. However, although they are generically included in resp. generically include the subspaces $ \Zm^{\ast}$ resp. $ \Sm^\ast $, i.e., generically $\Zm^g\subseteq\Zm^\ast$ resp. $ \Sm^\ast\subseteq\Sm_g $, the sets $\Zm^g$ and $\Sm_g$ are not controlled resp. conditioned invariant, in general, while the sets of nodes $ \Zm^\circ $ resp. $ \Sm^\circ $ are the maximal resp. the minimal sets of nodes with such properties.
\end{remark}

\section{Disturbance decoupling problems over networks: formulation and solution}
\label{sec:standard-DDP}

For a given network system \eqref{eq:lin-syst1}, we can consider two additional sets of nodes $ \Dm, \Tm \subset \mathcal{V}$, of cardinality $ | \Dm|$ and $ | \Tm | $, corresponding, respectively, to nodes affected by disturbances and target nodes. Each disturbance $ w_i $, $ i\in\{1, \ldots, |\Dm| \}$, is assumed to act independently on a node of $ \Dm$.
Both $ \Dm $ and $ \Tm $ can be identified with elementary matrices $ D \in \mathbb{R}^{n\times | \Dm|} $ and $ T\in \mathbb{R}^{| \Tm|\times n} $, i.e., matrices composed of elementary columns, respectively rows, representing subselections of nodes in $ \Vm$. 
We can then associate to $ \Dm $ and $ \Tm $ vector subspaces, respectively $ \im D =\Dm $ and $ \ker T = \Vm\smallsetminus\Tm$.

Denoting $ w = [ w_1 \ldots w_{|\Dm|}]^\top $, the complete system is therefore the following generalization of \eqref{eq:lin-syst1}:
\beq
\begin{split}
\dot x & =  Ax + B u + D w, \\
y & = C x, \\
z & = Tx. 
\end{split}
\label{eq:lin-syst2}
\eeq
For this network system, the disturbance decoupling problem studied in the literature can be formulated as follows.
\begin{problem}[\textbf{DDP}]\label{problem:DDP0}
Given the system~\eqref{eq:lin-syst2}, find a control law $u(t)$ that renders the target nodes in $\mathcal{T}$ unaffected by the disturbances in $\mathcal{D}$.
\end{problem}


Sufficient conditions for the solvability of Problem~\ref{problem:DDP0} for~\eqref{eq:lin-syst2} can be found in a purely topological way, independent of the weight values of the edges in $A$.

The construction of the control law $ u(t) $ requires instead to specify a feedback law, which in turn requires to know the numerical entries of $A$. The feedback can be constructed differently depending on the context. Common choices in the literature are:
\begin{itemize}
    \item \textbf{DDPSF} - Disturbance Decoupling Problem by State Feedback : if $u(t) = -Fx(t)$;
    \item \textbf{DDPOF} - Disturbance Decoupling Problem by Output Feedback: if $u(t) = -Gy(t)$;
    \item \textbf{DDPDF} - Disturbance Decoupling Problem by Dynamical Feedback: if $u(t) = -M\hat{x}(t) - Gy(t)$, where $\hat{x}(t)$ denotes the state estimate.
\end{itemize}

Obviously, the solvability of the DDPOF (or of the DDPDF) implies the solvability of the DDPSF, but not necessarily vice versa.

In Propositions~\ref{prop:Z} and~\ref{prop:S} we assume the following initializations.
\begin{assumption}\label{ass:1}
The algorithm for $\Zm^\circ$, presented in Proposition~\ref{prop:Z}, is initialized with $\Zm_0=\ker T=\Vm\smallsetminus\Tm$. 
\end{assumption}
Under Assumption~\ref{ass:1}, $\Zm^\circ$ is the maximal controlled invariant subset contained in $\Vm\smallsetminus\Tm$, that is, the largest set of nodes contained in $\Vm\smallsetminus\Tm$ which can be rendered invariant by an appropriate choice of state feedback. In other words, $ \Zm^\circ $ is rendered $ (A-BF)$-invariant for a suitable choice of a friend $F\in\mathbb{R}^{m\times n}$.
\begin{assumption}\label{ass:2}
The algorithm for $\Sm^\circ$, presented in Proposition~\ref{prop:S}, is initialized with $\Sm_0=\im D=\Dm$. 
\end{assumption}
Under Assumption~\ref{ass:2}, $\Sm^\circ$ is the minimal conditioned invariant superset containing $\Dm$, i.e., the smallest set of nodes containing $\Dm$ which can be rendered invariant by a suitable ``output injection''. We also say that $ \Sm^\circ $ is $ (A-HC) $-invariant for a proper choice of a friend $H\in\mathbb{R}^{n\times p}$. 

In order to avoid degenerate situations in the problem, we also make assumptions on the sets $\Dm$, $\Tm$, $\Vmin$ and $\Vmout$.
\begin{assumption}\label{ass:3}
$\Dm\cap\Tm = \emptyset $, $\Dm\cap\Vmin=\emptyset$ and $\Tm\cap\Vmout=\emptyset$.
\end{assumption}
In matrix form, the conditions of Assumption~\ref{ass:3} can be expressed as $TD=0$, $B^\top D=0$ and $TC^\top=0$.

\subsection{Solvability of the DDPSF}
If we consider the DDPSF, then we can say that the system~\eqref{eq:lin-syst2} becomes decoupled from the disturbances if $ \exists $ $u = -Fx$, s.t. the target vector
\begin{align*}
z(t) &= Te^{(A-BF)t}x_0+\int_0^{t}Te^{(A-BF)(t-\tau)}Dw(\tau)d\tau
\end{align*}
does not depend on the disturbance vector $w(t)$, $\forall\,t\geq0$. Equivalently, the DDPSF is solvable if $ \exists \, F$ s.t. $T(sI-A+BF)^{-1}D=0$, $\forall\,s$. For subspaces, the geometric solution to the DDPSF is well-known from the literature \cite{basile1969controlled,wonham1970decoupling}, and reads $\Dm\subseteq\Zm^\ast$, where $\Zm^\ast$ is the maximal controlled invariant subspace contained in $\ker T$. A constructive sufficient condition for the solvability of the DDPSF can be achieved by means of the maximal controlled invariant node set contained in $\Vm\smallsetminus\Tm$, i.e., $\Zm^\circ$.  



\begin{theorem}[\textbf{DDPSF}]\label{thm:ddpsf}
Under Assumptions~\ref{ass:1} and \ref{ass:3}, the DDPSF is solvable for the system~\eqref{eq:lin-syst2} if any of the following equivalent conditions is met:
\benu[label=(\alph*)]
\item $\Dm\subseteq\Zm^\circ(\Vmin)$;\label{eq:DDPSF-iff}
\item  Let $\Pm_1,\dots,\Pm_\ell$ be all possible $\Dm$-to-$\Tm$ paths. $\exists\,\Vmin$ s.t. $\Vmin\cap\Pm_j\neq \emptyset$, $\forall\,j\in\{1,\dots,\ell\}$.\label{item:c_thm:ddpsf}
\eenu
\end{theorem}


\subsection{Solvability of the DDPOF}
The DDPSF requires the whole $ x$ (i.e., the state of all nodes) to be measured, which may not be a realistic assumption for complex networks. An alternative is to select a subset $ \Vmout$ of nodes as measurable outputs and consider instead a DDPOF. The DDPOF is solvable if $\exists\,G\in\mathbb{R}^{m\times p}$ s.t. $T(sI-A+BGC)^{-1}D=0$, $\forall\,s$. The standard geometric solution to the DDPOF, given in Theorem 3.2 of \cite{hamano1975localization}, relies on the existence of an $(A,B)$-invariant and $(C,A)$-invariant subspace $\Wm$ s.t. $\Dm\subseteq\Wm\subseteq\Vm\smallsetminus\Tm$. A constructive sufficient condition for the solvability of the DDPOF can also be achieved by means of controlled and conditioned invariant node sets.

\begin{theorem}[\textbf{DDPOF}]\label{thm:ddpof}
Under Assumptions~\ref{ass:1}--\ref{ass:3}, the DDPOF is solvable for the system~\eqref{eq:lin-syst2} if any of the following equivalent conditions is met:
\benu[label=(\alph*)]
\item $\exists\,\Wm$ $(A,B)$-invariant and $(C,A)$-invariant set of nodes s.t. $\Dm\subseteq\Wm\subseteq\Vm\smallsetminus\Tm$;\label{item:a_thm:ddpof}
\item  Let $\Pm_1,\dots,\Pm_\ell$ be all possible $\Dm$-to-$\Tm$ paths. For every path $\Pm_j=\{v_1,\dots,v_{\kappa_j}\}$ of length $\kappa_j-1\geq1$, with $v_1\in\Dm$ and $v_{\kappa_j}\in\Tm$, $\exists$ at least one sub-path of length 1 of the form $\{v_{p},v_{p+1}\}$ with $v_{p}\in\Vmout$ and $v_{p+1}\in\Vmin$, $p\in\{1,\dots,\kappa_j-1\}$, $j\in\{1,\dots,\ell\}$.\label{item:c_thm:ddpof}
\eenu
\end{theorem}

\begin{corollary}\label{cor:ddpof}
$\Sm^\circ(\Vmout)\subseteq\Zm^\circ(\Vmin)$ is a necessary condition for the existence of a set of nodes $\Wm$ as in Theorem~\ref{thm:ddpof}.
\end{corollary}

\subsection{Solvability of the DDPDF}

Solving a DDPDF requires the computation of a compensator whose design is based on a suitable $(C,A,B)$-pair of subspaces $( \Sm, \, \Zm)$ s.t. $\Dm\subseteq\Sm\subseteq\Zm\subseteq\Vm\smallsetminus\Tm$. This is revised in Appendix~\ref{sec:compensator-design}. For subspaces, the geometric solution to the DDPDF reads $\Sm^\ast\subseteq\Zm^\ast$, where $\Sm^\ast$ is the minimal conditioned invariant subspace containing $\Dm$ and $\Zm^\ast$ is the maximal controlled invariant subspace contained in $\ker T$ (Theorem~6.6 and Corollary~6.7 of \cite{trentelman2012control}). Equivalently, $\exists$ an $A_c$-invariant subspace $\Wm_c$ of the extended state-space $\Xm\times \hat{\Xm}$, $\hat{\Xm}\subseteq\Xm=\mathbb{R}^n$, s.t. $\im D_c\subseteq\Wm_c\subseteq\ker T_c$ (Theorem~4.6 of \cite{trentelman2012control} on the closed-loop system \eqref{eq:closed}, see Appendix~\ref{sec:compensator-design} for the notation $ A_c, D_c, T_c$, and $ \hat{\Xm}$). The following theorem provides graphical sufficient conditions for the solvability of the DDPDF by means of the node sets defined in Propositions~\ref{prop:Z} and \ref{prop:S}.

\begin{theorem}[\textbf{DDPDF}]\label{thm:ddpdf}
Under Assumptions~\ref{ass:1}--\ref{ass:3}, the DDPDF is solvable for the systems~\eqref{eq:lin-syst2} and \eqref{eq:closed} if any of the following equivalent conditions is met:
\benu[label=(\alph*)]
\item $\Sm^\circ(\Vmout)\subseteq\Zm^\circ(\Vmin)$;\label{item:c_thm:ddpdf}
\item Let $\Pm_1,\dots,\Pm_\ell$ be all possible $\Dm$-to-$\Tm$ paths. For every path $\Pm_j=\{v_1,\dots,v_{\kappa_j}\}$ of length $\kappa_j-1\geq1$, with $v_1\in\Dm$ and $v_{\kappa_j}\in\Tm$, the indexes
\begin{align*}
  o_j=&\min\{s\in\{1,\dots,\kappa_j-1\}\;\textnormal{s.t.}\;v_s\in\Vmout\cap\Pm_j\},\\
  i_j=&\max\{s\in\{2,\dots,\kappa_j\}\;\textnormal{s.t.}\;v_s\in\Vmin\cap\Pm_j\},
\end{align*}
exist and are s.t. $o_j<i_j$, $\forall\,j\in\{1,\dots,\ell\}$.\label{item:d_thm:ddpdf}
\eenu
\end{theorem}

\begin{remark}
The choice of taking $B$ and $C$ elementary is a design strategy that translates in the freedom of assigning actuators and sensors to the nodes of the network. The condition of having $D$ and $T$ elementary is instead a simplification of the problem treatment. In fact one could relax the assumption that $D$ and $T^\top$ have only elementary columns and define the node sets $\Dm=\{v_i\in\Vm\;\text{s.t.}\;D_{ij}\neq0,\;\forall\,j\}$ and $\Tm=\{v_i\in\Vm\;\text{s.t.}\;T_{ji}\neq0,\;\forall\,j\}$. It follows that $\im D\subseteq\Dm$ and $\im T^\top\subseteq\Tm$, i.e., $\Vm\smallsetminus\Tm\subseteq\ker T$. Therefore, if we consider two sets of nodes $\Zm^\circ,\Sm^\circ\subseteq\Vm$ s.t. conditions $\Dm\subseteq\Zm^\circ$ and $\Sm^\circ\subseteq\Vm\smallsetminus\Tm$ hold true, automatically $\im D\subseteq\Zm^\circ$ and $\Sm^\circ\subseteq\ker T$. This means that the solvability conditions for the DDP of Theorems~\ref{thm:ddpsf}, \ref{thm:ddpof} and \ref{thm:ddpdf} are also sufficient for $D$ and $T$ non-elementary, provided that we adopt the above definitions for the node sets $\Dm$ and $\Tm$.
\end{remark}

\section{Minimal node cardinality DDP} 
\label{sec:min-node-DDP}

For the system~\eqref{eq:lin-syst2},  we consider now the following setting:
\bite
\item $ \Gm $ (and hence $A$) is given;
\item The sets of disturbances $ \Dm $ and targets $ \Tm $ are also given;
\item The sets of inputs $ \Vmin $ and outputs $ \Vmout $ are not preassigned.
\eite
The task is to choose $ \Vmin $ (and possibly $ \Vmout $) in the most parsimonious way while solving the DDP. 
More formally, the problem that we want to study is the following.

\begin{problem}[\textbf{Minimal node cardinality DDP}]\label{problem:DDP}
Given a network $\Gm$ and sets $ \Dm$, $ \Tm$, find a minimal cardinality set of input nodes $\Vmin$ (and, if needed, also a minimal cardinality set of output nodes $ \Vmout$) for which there exists a control law $u(t)$ that solves the DDP.
\end{problem}

In practice, solving Problem~\ref{problem:DDP} corresponds to finding a matrix $ B \in \mathbb{R}^{n\times m}$ of elementary columns s.t. $ \im B =\Vmin $ and, if needed, an elementary row matrix $ C \in \mathbb{R}^{p\times n} $ s.t. $ (\ker C)^\perp = \Vmout$. 
From the classification of Section~\ref{sec:standard-DDP}, the following subproblems can be specified.

\begin{problem2A}[\textbf{Minimal Input Cardinality DDPSF}]
\label{problem:DDPSF}
Find the minimal cardinality $ \Vmin$ that solves the DDPSF.
\end{problem2A}
\begin{problem2B}[\textbf{Minimal Input and Output Cardinality DDPOF}]
\label{problem:DDPOF}
Find the minimal cardinality $\Vmin$ and $ \Vmout$ that solves the DDPOF. 
\end{problem2B}
\begin{problem2C}[\textbf{Minimal Input and Output Cardinality DDPDF}]
\label{problem:DDPDF}
Find the minimal cardinality $\Vmin$ and $ \Vmout$ that solves the DDPDF.   
\end{problem2C}

\subsection{Solving the minimal input cardinality DDPSF}

To solve our minimal input cardinality DDPSF, we exploit condition~$\ref{eq:DDPSF-iff}$ of Theorem~\ref{thm:ddpsf}. For that we need to compute $ \Zm^\circ$ using the algorithm in Proposition~\ref{prop:Z}. Once $\Dm$ and $\Tm$ are assigned, and under Assumption~\ref{ass:1}, $ \Zm^\circ $ is only a function of $\Vmin$, i.e., $ \Zm^\circ =  \Zm^\circ (\Vmin)$.
Problem~2A can then be reformulated as an optimization problem: 
\beq
\text{Problem 2A:} \qquad \Biggl\{
    \begin{aligned}
      & \min_{\Vmin} \quad |\Vmin|\\
      & \text{subject to:} \quad \Dm\subseteq\Zm^\circ (\Vmin)
    .\end{aligned}\label{eq:opt_problem1}
\eeq



\subsubsection{A node min-cut/max-flow algorithm}\label{sec:mincut}
What item~$\ref{item:c_thm:ddpsf}$ of Theorem~\ref{thm:ddpsf} suggests is that the nodes of $ \Vmin$ must form a cut set of nodes separating $ \Dm $ from $ \Tm$. What Problem~2A requires is that such cut set must have minimal cardinality.
The problem can be seen as a special case of a min-cut/max-flow problem, whose optimal solution can be computed in polynomial time \cite{goldberg1988new}. 
The strategy to achieve this decoupling with the minimal number of control inputs is summarized in Algorithms~\ref{alg:one}--\ref{alg:three}.

\begin{algorithm}[hbt!]
\caption{DDPSF by min-cut/max-flow}\label{alg:one}
\begin{algorithmic}
\STATE \textbf{Data:} {$\mathcal{G}$, $\mathcal{D}$, $\mathcal{T}$}
\STATE \textbf{Result:} {$\Vmin$}
\STATE \underline{$[\overline{\mathcal{G}};v^\Dm; v^\Tm]\gets$  \textbf{Algorithm~\ref{alg:two}}};\hspace{8pt} \textit{// Given a network $\mathcal{G}$ with $n$ nodes and $q$ edges and the nodes subsets $\mathcal{D}$ and $\mathcal{T}$, construct an extended network $\overline{\mathcal{G}}$}
\vspace{0.1cm}
\STATE \underline{$[\overline{\Jm}^\Dm;\overline{\Jm}^\Tm]\gets$ \texttt{maxflow}($\overline{\mathcal{G}},v^\Dm,v^\Tm$)};\hspace{8pt} \textit{// Using the max-flow/min-cut algorithm from $v^\Dm$ to $v^\Tm$, partition the nodes of $\overline{\mathcal{G}}$ into the nodes that are and are not reachable from $v^\Dm$, resp. $\overline{\Jm}^\Dm$ and $\overline{\Jm}^\Tm$}
\STATE \underline{$\Vmin\gets$ \textbf{Algorithm~\ref{alg:three}}}.\hspace{8pt} \textit{// Given $\overline{\Jm}^\Dm$ and $\overline{\Jm}^\Tm$, find the optimal control input nodes set $\Vmin$ that solves the DDPSF}
\end{algorithmic}
\end{algorithm}

\begin{algorithm}[hbt!]
\caption{Construct the extended network $\overline{\mathcal{G}}$ }\label{alg:two}
\begin{algorithmic}
\STATE \textbf{Data:} {$\mathcal{G}$, $\mathcal{D}$, $\mathcal{T}$}
\STATE \textbf{Result:} {$\overline{\mathcal{G}}= ( \overline{\Vm}, \overline{\Em}, \overline{A}) $, $v^\Dm$, $v^\Tm$}
\STATE Construct $ \overline{\Vm}$: $\overline{\Vm}=\Vm\cup\{v_{n+1}, \, \ldots, v_{2n}\}\cup\{v_{2n+1},v_{2n+2}\}$;
\vspace{-0mm}
\STATE Construct $ \overline{\Em}$: 
\bite
\item If $\exists\,(v_i,\,v_j)\in\Em$ then $\exists\,(v_{i+n},\,v_{j}) \in\overline{\Em}$;
\vspace{-0mm}
\item $(v_i,\,v_{i+n})\in\overline{\Em},\,\forall\,i\in\{1,\dots,n\}$;
\vspace{-0mm}
\item $(v_{2n+1},\,v_{d})\in\overline{\Em},\,\forall\,v_d\in\Dm$;
\vspace{-0mm}
\item $(v_t,\,v_{2n+2})\in\overline{\Em},\,\forall\,v_t\in\Tm$;
\vspace{-0mm}
\eite
\STATE Construct $ \overline{A}$:
\vspace{-0mm}
\bite
\item If $\exists\,(v_i,\,v_j)\in\Em$ then $ \overline{A}_{j,i+n} =+\infty$;
\vspace{-0mm}
\item $ \forall \; v_i \notin \Dm $, $ \overline{A}_{i+n, i} =1 $;
\vspace{-0mm}
\item $ \forall \; v_i \in \Dm $, $ \overline{A}_{i+n, i} =+\infty $ and $\overline{A}_{i,2n+1}=+\infty$;
\vspace{-0mm}
\item $ \forall v_i \in \Tm $, $\overline{A}_{2n+2,i}=+\infty$;
\vspace{-0mm}

\eite
\STATE $ v^\Dm = v_{2n+1}$; $ v^\Tm = v_{2n+2}$.
\vspace{0mm}
\end{algorithmic}
\end{algorithm}

\begin{algorithm}[hbt!]
\caption{Minimal control input set $\Vmin$}\label{alg:three}
\begin{algorithmic}
\STATE \textbf{Data:} {$\overline{\Gm}$, $\overline{\Jm}^\Dm$, $\overline{\Jm}^\Tm$}
\STATE \textbf{Result:} {$\Vmin$}
\STATE $\Vmin=[\,]$;\hspace{8pt} \textit{// Store into $\Vmin$ the tail nodes of the edges between $\overline{\Jm}^\Dm$ and $\overline{\Jm}^\Tm$}
\FOR{$i=1:|\overline{\Jm}^\Dm|$}\FOR{$j=1:|\overline{\Jm}^\Tm|$}
\IF{$(\overline{\Jm}^\Dm(i),\overline{\Jm}^\Tm(j)) \in \overline{\Em}$}
\STATE{$\Vmin=\Vmin \cup \overline{\Jm}^\Dm(i)$;\;}
\ENDIF\ENDFOR\ENDFOR
.
\end{algorithmic}
\end{algorithm}

The algorithms always converge in finite time to an optimal solution, since the number of nodes and edges is finite and since, in the worst case scenario, the entire set $\Vm\smallsetminus\Dm$ can be exploited as control input nodes.



Here are more details on some parts of the algorithms.

\noindent {\bf Algorithm~\ref{alg:two}: Construct the extended network $\bm{\overline{\mathcal{G}}}$.}
The minimal node-cut problem can be transformed into a minimal edge-cut problem on an extended network $\overline{\Gm}=(\overline{\Vm},\overline{\Em},\overline{A})$, where it can be solved via a standard min-cut/max-flow algorithm. 
To do so, for each node $v_i$ a new node $v_{i+n}$ and a new edge $(v_i,\,v_{i+n})$ are generated. In all preexisting edges the weight is put at infinity (practically, at a very high number), whereas the additional edges $(v_i,\,v_{i+n})$ are assigned a unitary weight, except when $ v_i \in \Dm$, where $ (v_i, v_{i+n} )$ gets an infinite weight. 
In such a way, an edge-cut involving $(v_i,\,v_{i+n})$ corresponds to selecting the node $v_i$. Two additional nodes must be added, $v^\Dm=v_{2n+1}$ and $v^\Tm=v_{2n+2}$, so that $|\overline{\Vm}|=2n+2$. Node $v^\Dm$ acts as a source node and is connected via $|\Dm|$ edges to all disturbance nodes in $\Dm$. All target nodes in $\Tm$ are connected via $|\Tm|$ edges to the sink node $v^\Tm$. The total number of edges in $\overline{\Em}$ is $q+n+|\Dm|+|\Tm|$. On $ \overline{\Gm} $ all paths from $ v^\Dm $ to $ v^\Tm $ have odd length, i.e., must contain an even number of nodes.




\noindent {\bf Min-cut/Max-flow.}
A standard min-cut/max-flow algorithm \cite{ford1956maximal} is performed on $\overline{\Gm}=(\overline{\Vm},\overline{\Em}, \overline{A})$, from the source node $ v^\Dm $ to the sink node $ v^\Tm $. The cut partitions the network nodes in two disjoint sets, one, $ \overline{\Jm}^\Dm$, containing the nodes that are reachable from the disturbances, the other, $ \overline{\Jm}^\Tm $, containing the targets and all the nodes that cannot be reached by the disturbances.

\noindent {\bf Algorithm~\ref{alg:three}: find the minimal control input set $\Vmin$.}  
The meaning of Algorithm~\ref{alg:three} is explained in the following proposition.
\begin{proposition}
\label{prop-cut-extended}
Given $ \Gm $ and its extended network $ \overline{\Gm}$, an edge cut set in $ \overline{\Gm}$ corresponds to a node cut set in $ \Gm$. The min-cut value on $ \overline{\Gm} $ is equal to $ |\Vmin|$, i.e., to the cardinality of the set of input nodes associated with the solution.
\end{proposition}

In general the optimal solution $\Vmin$ provided by Algorithms~\ref{alg:one}--\ref{alg:three} is not unique. In order to obtain all optimal solutions, specialized techniques and extensions of the min-cut/max-flow algorithm (e.g., residual graph analysis, Gomory-Hu trees \cite{gomory1961multi}) can be used to enumerate all equal-cardinality cuts \cite{abboud2023all}. To discriminate among such optima, a secondary objective function must be introduced.

\subsubsection{Optimality}

The optimality of the solution is summarized in the following statement.
\begin{theorem}\label{thm:main_paper2}
Under Assumptions~\ref{ass:1} and \ref{ass:3}, the solutions $\Vmin$ computed by means of Algorithms~\ref{alg:one}--\ref{alg:three} are the optimal solutions of Problem~2A.
\end{theorem}

For the Problem~2A, we say that $ \Vmin$ is feasible if it satisfies $ \Dm \subseteq \Zm^\circ(\Vmin)$ but it is not necessarily minimal. 
The following characterization explains where the input nodes should be placed if we want to achieve optimality. 
Denote by $ \mathcal{P} = \bigcup_i \mathcal{P}_i $ the union of all $ \Dm$-to-$ \Tm$ paths. 
\begin{theorem}\label{thm:SF_otp_charact}
Under Assumptions~\ref{ass:1} and \ref{ass:3},
\benu[label=(\alph*)]
\item If $ \Vmin$ is a feasible solution of Problem 2A then \\$\Vmin \supseteq \partial_+(\Zm^\circ(\Vmin),A)$;\label{item:a_SF_otp_charact}
\item If $ \Vmin$ is an optimal solution of Problem 2A then \\$\Vmin=\partial_+(\Zm^\circ(\Vmin),A) \cap \mathcal{P} =\partial_+(\Zm^\circ(\Vmin),A) $;\label{item:b_SF_otp_charact}
\item $ \Vmin$ is an optimal solution of Problem 2A iff it is the minimal cardinality set of nodes for which $\Vmin=\partial_+(\Zm^\circ(\Vmin),A) \cap \mathcal{P} =\partial_+(\Zm^\circ(\Vmin),A) $.\label{item:c_SF_otp_charact}
\eenu
\end{theorem}

In words, an optimal solution cannot have inputs inside the maximal controlled invariant set $ \Zm^\circ$, meaning that an optimal solution must be the out-boundary of $ \Zm^\circ$. As a consequence, an optimal solution always satisfies $\Vmin\cap\Zm^\circ(\Vmin)=\emptyset$.
In Theorem~\ref{thm:SF_otp_charact}, the condition in~$\ref{item:b_SF_otp_charact}$ is necessary but not sufficient for optimality, since there may exist other cut-set solutions satisfying $\Vmin=\partial_+(\Zm^\circ(\Vmin),A)=\partial_+(\Zm^\circ(\Vmin),A) \cap \mathcal{P}$ but of higher cardinality.
Note that $\Vmin=\partial_+(\Zm^\circ(\Vmin),A)$ alone does not imply that $\Vmin$ are placed on the $\Dm$-to-$\Tm$ paths only. 
The non-minimality of a feasible solution is also expressed in the following proposition.
\begin{proposition}
\label{thm:sub}
Under Assumptions~\ref{ass:1} and \ref{ass:3}, and given a feasible solution $\Vmin=\im B$ of Problem~2A of cardinality $m$, define the submatrix $\Tilde{B}=B\,\diag(\ell_1,\dots,\ell_m)$, where $\ell_i\in\{0,1\},\,\forall\,i$, with $\ell_i\neq0\;\forall\,v_i\in\partial_+(\Zm^\circ(\Vmin),A)$. Then $ \Zm^\circ (\Vmin)$ is $(A,\Tilde{B})$-invariant.
\end{proposition}

\subsection{Solving the minimal input/output cardinality DDPOF and DDPDF}
The equivalent of Problem~2A for the DDPOF is the following.
\beq
\text{Problem 2B:} \qquad \Biggl\{
    \begin{aligned}
      & \min_{\Vmin,\Vmout} \quad |\Vmin|+|\Vmout|\\
      & \text{subject to:} \;\; \Dm\subseteq\Wm\subseteq\Vm\smallsetminus\Tm,
    \end{aligned}\label{eq:opt_problem2B}
\eeq
where $\Wm$ is the $(A,B)$- and $(C,A)$-invariant set of nodes in item~$\ref{item:a_thm:ddpof}$ of Theorem~\ref{thm:ddpof}.

The minimal input/output cardinality DDPDF can instead be formulated as follows:
\beq
\text{Problem 2C:} \qquad \Biggl\{
    \begin{aligned}
      & \min_{\Vmin,\Vmout} \quad |\Vmin|+|\Vmout|\\
      & \text{subject to:} \quad \Sm^\circ(\Vmout)\subseteq\Zm^\circ (\Vmin)
    .\end{aligned}\label{eq:opt_problem2C}
\eeq

Similarly to Theorem~\ref{thm:SF_otp_charact}, we have the following characterizations of the feasibility and optimality of the DDPOF and the DDPDF.

\begin{theorem}\label{thm:OF_otp_charact}
Under Assumptions~\ref{ass:1}--\ref{ass:3},
\benu[label=(\alph*)]
\item If $ \Vmin$, $\Vmout$ are a feasible solution of Problem 2B then \\$\Vmin \supseteq \partial_+(\Wm,A)$ and $\Vmout \supseteq \partial_-(\Wm,A)$;\label{item:a_OF_otp_charact}
\item If $ \Vmin$, $\Vmout$ are an optimal solution of Problem 2B then \\$\Vmin = \partial_+(\Wm,A)\cap \mathcal{P}=\partial_+(\Wm,A)$ and $\Vmout = \partial_-(\Wm,A)\cap \mathcal{P}=\partial_-(\Wm,A)$;\label{item:b_OF_otp_charact}
\item $ \Vmin$, $\Vmout$ are an optimal solution of Problem 2B iff they are the minimal cardinality sets of nodes for which $\Vmin = \partial_+(\Wm,A)\cap \mathcal{P}=\partial_+(\Wm,A)$ and $\Vmout = \partial_-(\Wm,A)\cap \mathcal{P}=\partial_-(\Wm,A)$.\label{item:c_OF_otp_charact}
\eenu
\end{theorem}

\begin{theorem}\label{thm:DF_otp_charact}
Under Assumptions~\ref{ass:1}--\ref{ass:3},
\benu[label=(\alph*)]
\item If $ \Vmin$, $\Vmout$ are a feasible solution of Problem 2C then \\$\Vmin \supseteq \partial_+(\Zm^\circ(\Vmin),A)$ and $\Vmout \supseteq \partial_-(\Sm^\circ(\Vmout),A)$;\label{item:a_DF_otp_charact}
\item If $ \Vmin$, $\Vmout$ are an optimal solution of Problem 2C then \\$\Vmin = \partial_+(\Zm^\circ(\Vmin),A)\cap \mathcal{P}=\partial_+(\Zm^\circ(\Vmin),A)$ and $\Vmout = \partial_-(\Sm^\circ(\Vmout),A)\cap \mathcal{P}=\partial_-(\Sm^\circ(\Vmout),A)$;\label{item:b_DF_otp_charact}
\item $ \Vmin$, $\Vmout$ are an optimal solution of Problem 2C iff they are the minimal cardinality sets of nodes for which $\Vmin = \partial_+(\Zm^\circ(\Vmin),A)\cap \mathcal{P}=\partial_+(\Zm^\circ(\Vmin),A)$ and $\Vmout = \partial_-(\Sm^\circ(\Vmout),A)\cap \mathcal{P}=\partial_-(\Sm^\circ(\Vmout),A)$.\label{item:c_DF_otp_charact}
\eenu
\end{theorem}


\section[Feedback synthesis]{Feedback synthesis}
\label{sec:feedback}

\subsection{State feedback for the minimal input cardinality DDPSF}

If computing the input set $ \Vmin$ that solves Problem~2A requires only the topology of the network, computing the control law $u(t) = - Fx $ that solves the DDPSF requires knowledge of the edge weight matrix $A$, and also of the state vector $ x$. 
In particular, using Proposition~\ref{prop:AB-invar-equiv-char}, a standard choice in the literature for the friend $ F$ is given by the following formula (see \cite{basile1969controlled})
\beq\label{eq:friend}
F=U(Z^\top Z)^{-1}Z^\top + F_q,
\eeq
where $ Z $ is the basis matrix associated with a subset $ \Zm $,
and $F_q=M(Z^\perp)^\top$, with $M$ being a $m\times(n-|\Zm|)$ matrix of coefficients and $Z^\perp$ a basis matrix of $\Zm^\perp$.
In fact, from \eqref{eq:friend} we obtain $(A-BF) Z = AZ-BU(Z^\top Z)^{-1}Z^\top Z+BM(Z^\perp)^\top Z=AZ-BU=ZX$,
which proves that $\Zm$ is $(A-BF)$-invariant, see Proposition~\ref{prop:AB-invar-equiv-char}. As $ F_q $ is by construction orthogonal to $ Z$, the design of a friend for $\Zm$ depends only on the term $U(Z^\top Z)^{-1}Z^\top$. In the context of controlled invariant sets of nodes, we have the following characterizations for $F$.

\begin{theorem}\label{thm:friend}
Under Assumptions~\ref{ass:1} and \ref{ass:3}, if $\Zm\subseteq\Vm$ is an $(A,B)$-invariant set of nodes, of basis matrix $Z$, then a friend of $ \Zm$ solving the DDPSF is $F=UZ^\top+F_q$, where $U$ is s.t. $AZ = Z X +BU$ holds for some $ X$.
Furthermore, if $\Vmin$ is a set of input nodes s.t. $\Vmin=\partial_+(\Zm^\circ(\Vmin),A)$, then $F=B^\top A+F_q$ is the unique friend of $\Zm=\Zm^\circ(\Vmin)$, up to $F_q=M(Z^\perp)^\top$.
\end{theorem}

\begin{remark}\label{rem:friend}
In an optimal solution of Problem~2A, which by item~$\ref{item:b_SF_otp_charact}$ of Theorem~\ref{thm:SF_otp_charact} is guaranteed to satisfy the condition $\Vmin=\partial_+(\Zm^\circ(\Vmin),A)$, the feedback action of $F$ can be interpreted graphically as follows. The term $B^\top A$ acts as an edge removal for all edges from the parents of the control nodes in $\Vmin$ to $ \Vmin$, modifying the adjacency matrix to $A  - BF = A - BB^\top A$, i.e., canceling the rows of $A$ identified by $\Vmin$. The term $F_q$ instead adds edges of arbitrary weight that have tails in $\Zm^{\circ\perp}$ and heads in $\Vmin$ (see Example~\ref{ex:ex1}).
\end{remark}


\subsection{Output feedback and dynamical feedback for the minimal input/output cardinality DDPOF and DDPDF}

The synthesis outlined by Theorem~\ref{thm:friend} suggests that a similar result can be found for the output injection $-Hy$. Under similar assumptions, it is shown that the friend $H$ that renders the conditional invariant set $\Sm^\circ(\Vmout)$, $(A-HC)$-invariant, admits an intuitive graphical interpretation.
\begin{theorem}    
\label{thm:output_inj}
Under Assumptions~\ref{ass:2} and \ref{ass:3}, and given a $(C,A)$-invariant set of nodes $\Sm^\circ(\Vmout)$, if $\Vmout$ is a set of output nodes s.t. $\Vmout=\partial_-(\Sm^\circ(\Vmout),A)$, then $H=AC^\top+H_q$ is the unique friend of $\Sm^\circ(\Vmout)$, up to $H_q=SN$ with $N$ being a generic $|\Sm|\times p$ matrix of coefficients and $S$ a basis matrix of $\Sm$.
\end{theorem}
As for the synthesis of the output feedback, we have the following result.
\begin{theorem}\label{thm:G}
Under Assumptions~\ref{ass:1}--\ref{ass:3}, and given a $(A,B)$- and $(C,A)$-invariant set of nodes $\Wm$, if $\Vmout$ and $\Vmin$ are two sets of output and input nodes, respectively, s.t. $\Vmout=\partial_-(\Wm,A)$, and $\Vmin=\partial_+(\Wm,A)$, then the unique friend of $\Wm$ solving the DDPOF, i.e., the matrix $G\in\mathbb{R}^{m\times p}$ s.t. $(A-BGC)\Wm\subseteq\Wm$, is given by $G=B^\top AC^\top$.
\end{theorem}

\begin{remark}\label{rem:friend_W}
Item~$\ref{item:b_OF_otp_charact}$ of Theorem~\ref{thm:OF_otp_charact} guarantees that the conditions $\Vmin=\partial_+(\Wm,A)$ and $\Vmout=\partial_-(\Wm,A)$ are satisfied for the optimal solutions of Problem~2B.
Graphically, the effect of $G=B^\top AC^\top$ is to remove the sub-paths of length 1 described in item~$\ref{item:c_thm:ddpof}$ of Theorem~\ref{thm:ddpof}.
\end{remark}
The results of Theorems~\ref{thm:friend}-\ref{thm:G} are used for the synthesis of the reduced order compensator for the DDPDF (see Appendix~\ref{sec:compensator-design} for a recap), as described in the next theorem.
\begin{theorem}\label{thm:COM}
Under Assumptions~\ref{ass:1}--\ref{ass:3}, and given a $(C,A,B)$-pair of node sets $(\Sm^\circ(\Vmout),\Zm^\circ(\Vmin))$, identified by the sets of output and input nodes, $\Vmout$ and $\Vmin$ respectively, if $\Vmout=\partial_-(\Sm^\circ(\Vmout),A)$ and $\Vmin=\partial_+(\Zm^\circ(\Vmin),A)$, then the reduced-order compensator $(PKP^\top,PL,MP^\top,G)$ of order $|\Zm^\circ|-|\Sm^\circ|$ solving the DDPDF is uniquely defined by
\[
\begin{cases}
K=A-BB^\top A-AC^\top C+BB^\top AC^\top C-K_q, \\
L = AC^\top-BB^\top AC^\top+H_q,\\
M = B^\top A-B^\top AC^\top C+F_q,\\
G = B^\top AC^\top,
\end{cases}
\]
 with $P$ being the projection matrix defined in Lemma~\ref{lem:proj} (in Appendix~\ref{sec:compensator-design}), and $K_q=BF_q+H_qC$ with $F_q$, $H_q$ defined in Theorems~\ref{thm:friend} and \ref{thm:output_inj}.
\end{theorem}

\begin{remark}
The effect of the dynamical feedback is to compensate for the flow in $\Vmin$ via a state estimate having non-zero components corresponding to the nodes in $\Zm^\circ(\Vmin)\smallsetminus\Sm^\circ(\Vmout)$ (see Example~\ref{ex:ex2}).
\end{remark}
\section{Numerical examples}
We start this numerical section by giving an example that, under Assumptions~\ref{ass:1} and \ref{ass:2}, illustrates the computation of the node sets $\Zm^\circ$ and $\Sm^\circ$ by means of the recursions of Propositions~\ref{prop:Z} and \ref{prop:S}.\\
\begin{example}
\label{ex:exx0}

Let us consider the network $ \Gm $ of $13$ nodes of Fig.~\ref{fig:exx01}. We assign $\Tm=\{v_{11},v_{12}\}$ and $\Bm=\{v_8\}$ and we note that at the second iteration of the recursion of Proposition~\ref{prop:Z} we converge to the maximal controlled invariant set of nodes contained in $\Vm\smallsetminus\Tm$, $\Zm^\circ(\{v_8\})=\{v_1,v_2,v_3,v_4,v_5,v_6,v_7\}$. Similarly, we assign $\Dm=\{v_1,v_5\}$ and $\Cm=\{v_3\}$ and we note that at the second iteration of the recursion of Proposition~\ref{prop:S} we converge to the minimal conditioned invariant set of nodes containing $\Dm$, $\Sm^\circ(\{v_3\})=\{v_1,v_2,v_3,v_4,v_5\}$. It is easy to see that since $\Dm\subset\Zm^\circ$, the DDPSF is solvable. Similarly, since $(\Sm^\circ,\Zm^\circ)$ constitutes a $(C,A,B)$-pair of node sets, also the DDPDF is solvable. However, we note that the conditions of Theorem~\ref{thm:ddpof} are not met, therefore the DDPOF cannot be solved with these choices of $\Bm$ and $\Cm$.
\begin{figure}[htb!]
    \centering
    \includegraphics[trim=0cm 0cm 0cm 0cm, clip=true, width=1\linewidth]{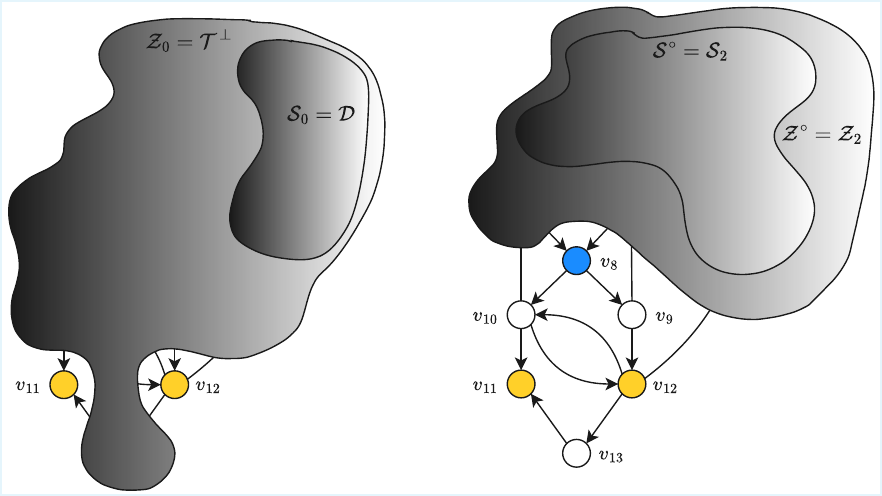}
    \caption{Example~\ref{ex:exx0}, computation of $\Zm^\circ$ and $\Sm^\circ$. Disturbances are in red, targets in yellow, control inputs in blue, whereas the output nodes are in green.}
    \label{fig:exx01}
\end{figure}
 \hfill \qed
\end{example}

The following example highlights the consequences of Theorem~\ref{thm:relat}.
\begin{example}\label{ex:ex0}
Consider the network depicted in Fig.~\ref{fig:S-Z}, with $\Dm=\{v_3\}$, $\Tm=\{v_5\}$, $\Cm=\{v_4\}$ and $\Bm=\{v_2\}$, and take a generic matrix $A=(a_{ij})\in\Am$.
\begin{figure}[h!]
    \centering    \includegraphics[width=0.3\linewidth]{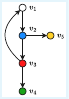}
    \caption{Example~\ref{ex:ex0}. Disturbances are in red, targets in yellow, control inputs in blue, whereas the output node is in green.}\label{fig:S-Z}
    \vspace{-0.1cm}
\end{figure}
By means of Propositions~\ref{prop:Z} and \ref{prop:S} we get $\Sm^\circ=\Vm$ and $\Zm^\circ=\{v_1,v_3,v_4\}$. Now we compute $\Sm^\ast$ by means of the classical recursion in Proposition~\ref{prop:cond_min_semi}. We note that $D=e_3$, $C=e_4^\top$, thus $\im D=\Sm_0=\Sm_0\cap\ker C=\textnormal{span}\{e_3\}$. Since $Ae_3=[a_{13}\;0\;0\;a_{43}\;0]^\top$ it follows that $\Sm_1=\textnormal{span}\{e_3\}+A(\Sm_0\cap\ker C)=\textnormal{span}\{e_3,[a_{13}\;0\;0\;a_{43}\;0]^\top\}$. By noting that also $\Sm_1\cap\ker C=\textnormal{span}\{e_3\}=\Sm_0\cap\ker C$, we conclude that $\Sm_2=\Sm_1=:\Sm^\ast$. From $\Sm^\ast=\textnormal{span}\{e_3,[a_{13}\;0\;0\;a_{43}\;0]^\top\}$ we note that not all the nodes in $\Sm^\circ$ appear as nonzero coordinates in the generators of $\Sm^\ast$, but that all the nonzero coordinates that appear in $\Sm^\ast$ are also nodes of $\Sm^\circ$, as proven in Theorem~\ref{thm:relat}. Furthermore, in this example  $\Zm^\ast=\Zm^\circ$. In fact, by applying the recursion of Proposition~\ref{prop:contr_max_semi}, we note that since $B=e_2$ and $T=e_5^\top$ we get $\Zm_0=\ker T=\textnormal{span}\{e_1,e_2,e_3,e_4\}=\Zm_0+\im B$. Now taking $x\in\mathbb{R}^5$ and imposing $Ax=[a_{13}x_3\;a_{21}x_1\;a_{32}x_2\;a_{43}x_3\;a_{52}x_2]^\top\in\Zm_0+\im B$ means $x_2=0$, since $a_{52}\neq0$. Thus $A^{-1}(\Zm_0+\im B)=\ker e_2^\top$, so $\Zm_1=\ker T\cap\ker e_2^\top=\ker[e_5\;e_2]^\top=\textnormal{span}\{e_1,e_3,e_4\}$. By noting that $\Zm_1+\im B=\Zm_0$, we conclude that $\Zm_2=\Zm_1=:\Zm^\ast$. We note that $\Dm\subseteq\Zm^\circ$, and therefore the DDPSF is solvable. Now note that $\nexists$ a set of nodes both controlled and conditioned invariant, hence the DDPOF based on the sets of nodes is not solvable. In fact, as proven in Theorem~\ref{thm:ddpof}, in all the $\Dm$-to-$\Tm$ paths (e.g., $\{v_3,v_1,v_2,v_5\}$) there should be at least a sub-path of length $1$ from $\Cm$ to $\Bm$. However, the DDPOF based on the subspaces is solvable in this case. We note that the subspace $\Sm^\ast=\text{span}\{e_3,[a_{13}\;0\;0\;a_{43}\;0]^\top\}$ is also a controlled invariant subspace contained in $\ker T$. The DDPOF is thus solvable for $\Wm=\Sm^\ast$ with the friend $G=\frac{a_{13}a_{21}}{a_{43}}$. Moreover, we note that $(\Sm^\circ,\Zm^\circ)$ is not a $(C,A,B)$-pair while $(\Sm^\ast,\Zm^\ast)$ is. Therefore the DDPDF is solvable using the $(C,A,B)$-pair $(\Sm^\ast,\Zm^\ast)$, and it is possible to compute a dynamic compensator (see Appendix~\ref{sec:compensator-design} for more details).
 \hfill \qed

\end{example}

In the following example, an application of Algorithms~\ref{alg:one}--\ref{alg:three} for the DDPSF is presented.
\begin{example}
\label{ex:ex1}
Let us consider the network $ \Gm $ of $7$ nodes and $\Dm~=~\{v_1,v_5\}$ and $\Tm=\{v_4,v_7\}$ shown in Fig.~\ref{fig:ex1a} on the left.
Using Algorithms~\ref{alg:one}--\ref{alg:three} the extended network $\overline{\Gm}$ shown in Fig.~\ref{fig:ex1a} on the right is constructed.
\begin{figure}[htb!]
    \centering
    \includegraphics[clip=true, width=.9\linewidth]{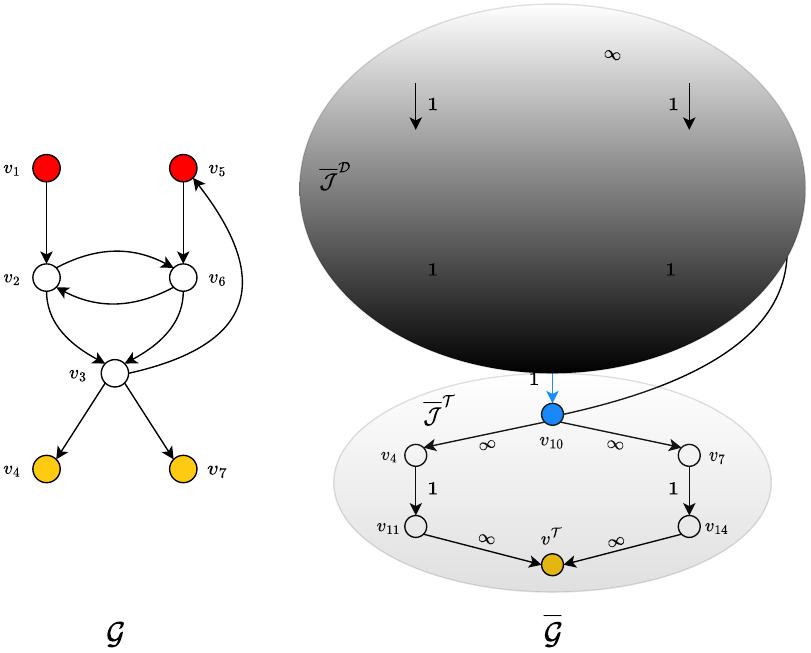}
    \caption{Example~\ref{ex:ex1}. Disturbances are in red, targets in yellow and the control node is in blue.}
    \label{fig:ex1a}
\end{figure}
The minimal cardinality solution of Problem~2A is $\Vmin=\{v_3\}$. In fact, the edge $(v_3,v_{10})$ defines a cut set of edges in the extended network $\overline{\Gm}$. The matrix $B^\top A+F_q=\begin{bmatrix}
0 & a_{3,2} & \ast & \ast & 0 & a_{3,6} & \ast
\end{bmatrix}$ is the unique friend of $\Zm^\circ(\{v_3\})=\{v_1,v_2,v_5,v_6\}$, as proven in Theorem~\ref{thm:friend}, where the symbol $*$ denotes an arbitrary scalar. 
Its effect is to cancel the edges $ ( v_2, v_3 )$ and $ (v_6, v_3 )$ and to add, with arbitrary weights, the edges $(v_4,v_3)$, $(v_7,v_3)$ and the self-loop in $v_3$. \hfill \qed
\end{example}
In the next example, the optimal solutions for the DDPOF and the DDPDF are compared.
\begin{example}
\label{ex:ex2}
Let us consider the network $ \Gm $ of $21$ nodes and $\Dm~=~\{v_1,v_2\}$ and $\Tm=\{v_{20},v_{21}\}$ of Fig.~\ref{fig:ex2a}. For it, the minimal cardinality solution of Problem~2B is $\Vmin=\{v_{14},v_{15}\}$ and $\Vmout=\{v_{12},v_{13}\}$, whereas one possible minimal cardinality solution of Problem~2C is $\Vmin=\{v_{14},v_{15}\}$ and $\Vmout=\{v_7\}$.
\begin{figure}[htb!]
    \centering
    \includegraphics[trim=0cm 0cm 0cm 0cm, clip=true, width=0.45\linewidth]{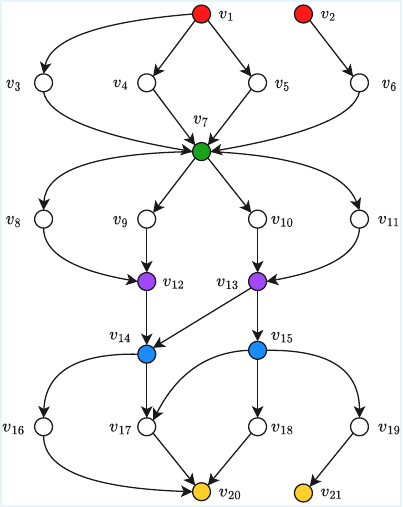}
    \caption{Example~\ref{ex:ex2}. Disturbances are in red, targets in yellow, control inputs in blue, whereas the output nodes are in purple for the DDPOF and in green for the DDPDF.}
    \label{fig:ex2a}
\end{figure}
We note that in the DDPOF $\Zm^\circ(\{v_{14},v_{15}\})=\Sm^\circ(\{v_{12},v_{13}\})=\Wm$, and that the unique friend for $\Wm$ is $B^\top AC^\top=\begin{bmatrix}
a_{14,12} & a_{14,13} \\
0 & a_{15,13}
\end{bmatrix}$, as proven in Theorem~\ref{thm:G}. Its effect is to remove the edges $(v_{12},v_{14})$, $(v_{13},v_{14})$ and $(v_{13},v_{15})$, i.e., the sub-paths of length 1 described in item~$\ref{item:c_thm:ddpof}$ of Theorem~\ref{thm:ddpof}. On the other hand, for the DDPDF the output node $v_7$ can be selected s.t. item~$\ref{item:d_thm:ddpdf}$ of Theorem~\ref{thm:ddpdf} is satisfied. The realization $A_c$ in the extended state-space induced by the reduced-order compensator described in Theorem~\ref{thm:COM}, with projection matrix $P=\begin{bmatrix}
    e_8\;e_9\;e_{10}\;e_{11}\;e_{12}\;e_{13}
\end{bmatrix}^\top$ onto $\Zm^\circ(\Vmin)\smallsetminus\Sm^\circ(\Vmout)=\{v_8,v_9,v_{10},v_{11},v_{12},v_{13}\}$, is given in Fig.~\ref{fig:ex2b}.
 \hfill \qed
\end{example}
\begin{figure}[htb]
    \centering
    \includegraphics[trim=0cm 0cm 0cm 0cm, clip=true, width=.9\linewidth]{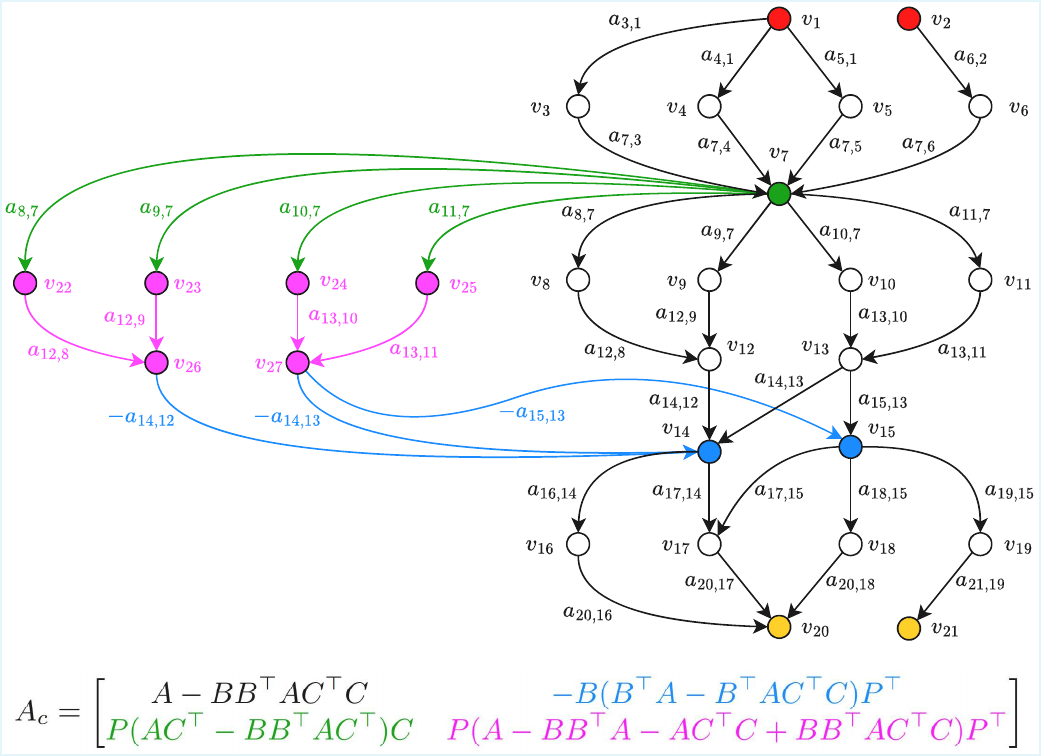}
    \caption{DDPDF of Example~\ref{ex:ex2}. Disturbances are in red, targets in yellow, control inputs in blue, output nodes in green whereas the nodes in the observer are in fuchsia. The colored edges reflect the blocks in the closed-loop realization $A_c$.}
    \label{fig:ex2b}
     \vspace{-0.4cm}
\end{figure}

\section{Conclusion}

The main purpose of this paper is to reformulate the key ideas of geometric control theory for networked systems. In particular the problem of protecting strategic nodes from the influence of disturbances is reformulated as a node-based DDP, and suitable feedback laws achieving the decoupling are obtained based on the node set notions of controlled and conditioned invariance. 
The node-based setting allows to formulate the minimal input/output node allocation as a graphical optimization problem, for which polynomial-complexity solutions are in many cases available.\\
Ongoing work focuses on accounting for closed-loop stability and, more generally, pole placement. Further work includes investigating the behavior of the proposed algorithms on different network topologies (see, e.g., \cite{hamilton2010patterned}) as well as on real-world datasets and problems. 


\appendices

\section{Maximal controlled invariance, and minimal conditioned invariance for subspaces}
\label{app:subsp-max-min}
In this Appendix we report the subspace counterpart of Propositions~\ref{prop:Z} and \ref{prop:S}.
\begin{proposition}[\textbf{Algorithm 4.1.2 of \cite{basile1992controlled}}]\label{prop:contr_max_semi}
Given a vector subspace $\Zm\subseteq\mathbb{R}^n$ and the recursion
\bite
\item$ \Zm_0 = \Zm$;
\item$ \Zm_{k+1} = \Zm\cap A^{-1}(\Zm_{k}+\im B)$,
\eite
$\exists$ an index $r\leq n-1$ s.t. $\Zm_{r+p}=\Zm_{r}$, $\forall\,p\geq1$. The subspace $\Zm^\ast\coloneqq\Zm_r$ is the maximal control invariant subspace contained in $\Zm$.
\end{proposition}
\begin{remark}
The term $A^{-1}(\Zm_{k}+\im B)$ denotes the preimage of $A$ of the subspace $\Zm_{k}+\im B$ which is well defined also when $A$ is not invertible and corresponds to $\{x\in\mathbb{R}^n\;\text{s.t.}\;Ax\in\Zm_{k}+\im B\}$.
\end{remark}
\begin{proposition}[\textbf{Algorithm 4.1.1 of \cite{basile1992controlled}}]\label{prop:cond_min_semi}
Given a vector subspace $\Sm\subseteq\mathbb{R}^n$ and the recursion
\bite
\item $ \Sm_0 = \Sm$;
\item $ \Sm_{k+1}= \Sm + A(\Sm_{k}\cap\ker C)$,
\eite
$\exists$ an index $r\leq n-1$ s.t. $\Sm_{r+p}=\Sm_{r}$, $\forall\,p\geq1$. The subspace $\Sm^\ast\coloneqq\Sm_r$ is the minimal conditioned invariant subspace containing $\Sm$.
\end{proposition}
\begin{remark}
In the context of DDP, the recursions of Propositions~\ref{prop:contr_max_semi} and \ref{prop:cond_min_semi} are initialized with $\Zm=\ker T$ and $\Sm=\im D$, respectively, cf. Assumptions~\ref{ass:1} and~\ref{ass:2}.
\end{remark}
\section{Proofs of Section~\ref{sec:geom-contr}}
\label{sec:app1}

\textbf{Proof of Proposition~\ref{prop:invA}}
Follows as a special case of Proposition~\ref{prop:AB-invar-equiv-char}.
\qed

\textbf{Proof of Proposition~\ref{prop:AB-invar-equiv-char}}
Items~$\ref{item:b_prop:AB-invar-equiv-char}$ and $\ref{item:c_prop:AB-invar-equiv-char}$ are standard for subspaces and hence also for sets of nodes by Definition~\ref{def:sub}, see, e.g., Theorem~4.2 of \cite{trentelman2012control} and Property~4.1.3 of \cite{basile1992controlled}.
To prove item $\ref{item:a_prop:AB-invar-equiv-char}$, note that $(v_i , \, v_j) \in \Em$ iff $A_{ji}\neq 0$. Moreover, the correspondence $v_i\sim e_i$ of Definition~\ref{def:sub} implies that $v_i\in\Zm$ (as a set of nodes) iff $ e_i\in\Zm$ (as a vector space). From $A_{ji}\neq 0$ it follows that $\mathbb{A}e_i=r_i+e_j$, where the vector $r_i\neq0$ iff $\exists \,k\neq j$ s.t. $A_{ki}\neq0$ (i.e., iff $\exists$ an edge starting in $ v_i $ landing in $ v_k$). Therefore, if we assume the $(A,B)$-invariance to hold, then $A\Zm\subseteq\Zm
+\Vmin\implies r_i+e_j\in\Zm+\Vmin$. Thus, since $\Zm$ and $\Vmin$ are generated by canonical basis vectors, from Definition~\ref{def:sub} we conclude that $e_j\in\Zm+\Vmin$, which implies that $v_j\in\Zm\cup\Vmin$. Vice versa, since $\mathbb{A} e_i=
\sum_{v_j \;\text{s.t.}\; (v_i,v_j)\in \Em} e_j$, if $v_i\in\Zm$ and $v_j\in\Zm\cup\Vmin$, then we obtain $\mathbb{A} e_i\in\operatorname{span}\{ e_j \;\text{s.t.}\; v_j \in \Zm \cup \Bm \}$. Thus $\mathbb{A} e_i \in\Zm+\Bm$, $\forall$ $v_i \in \Zm$. By noting that a generic~element $z$ of $\Zm$ is a linear combination of elementary vectors, we conclude that $\mathbb{A}z\in\Zm+\Vmin$, which implies that $A\Zm\subseteq\Zm+\Vmin$.
\qed

\textbf{Proof of Proposition~\ref{prop:CA-invar-equiv-char}}
Item~$\ref{item:b_prop:CA-invar-equiv-char}$ is standard for subspaces and hence also for sets of nodes by Definition~\ref{def:sub}, see, e.g., Theorem~5.5 of \cite{trentelman2012control}. To prove item $\ref{item:a_prop:CA-invar-equiv-char}$, note that $(v_i , \, v_j) \in \Em $ iff $ A_{ji}\neq 0$. Moreover, the correspondence $v_i\sim e_i$ of Definition~\ref{def:sub} implies that $e_i\in\Sm$ iff $ v_i\in\Sm$. 
From $A_{ji}\neq 0$ and taken $e_i\in\Sm\smallsetminus\Vmout$, it follows that $\mathbb{A}e_i=r_i+e_j$, where the vector $r_i\neq0$ iff $\exists \,k\neq j$ s.t. $A_{ki}\neq0$. Therefore, if we assume the $(C,A)$-invariance to hold, then $A(\Sm\smallsetminus\Vmout)\subseteq\Sm
\implies r_i+e_j\in\Sm$. Thus, since $\Sm$ is generated by canonical basis vectors, from Definition~\ref{def:sub} we conclude that $e_j\in\Sm$, which implies that $v_j\in\Sm$. Vice versa, since $\mathbb{A} e_i=
\sum_{v_j \;\text{s.t.}\; (v_i,v_j)\in \Em} e_j$, if $v_i\in\Sm\smallsetminus\Vmout$ and $v_j\in\Sm$, then we obtain $\mathbb{A} e_i\in\operatorname{span}\{ e_j \;\text{s.t.}\; v_j \in \Sm \}$. Thus $\mathbb{A} e_i \in\Sm$, $\forall$ $v_i \in \Sm\smallsetminus\Vmout$. By noting that a generic element $s$ of $\Sm\smallsetminus\Vmout$ is a linear combination of elementary vectors, we conclude that $\mathbb{A}s\in\Sm$, which implies that $A(\Sm\smallsetminus\Vmout)\subseteq\Sm$.
\qed

\textbf{Proofs of Propositions~\ref{prop:Z} and \ref{prop:S}}
The proof of Proposition~\ref{prop:Z} follows by duality from the one of Proposition~\ref{prop:S}, using Proposition~\ref{prop:dual_net}. To prove Proposition~\ref{prop:S}, we first note that $\Dm\subseteq\Sm^\circ$ by construction. Assume by contradiction that $\Sm^\circ$ is not $(C,A)$-invariant, then there must exist at least one node $v_i$ in $\Sm^\circ\smallsetminus\Cm$, a node $v_j$ in $\Sm^{\circ\perp}$ and an edge $(v_i,v_j)$. This contradicts the definition of conditioned invariant node set, given in item~$\ref{item:a_prop:CA-invar-equiv-char}$ of Proposition~\ref{prop:CA-invar-equiv-char}. Minimality is trivial by the construction of the scheme, while the convergence in at most $n$ steps is guaranteed by the fact that the sequence is non-decreasing and the total number of nodes is bounded by $n$.
\qed

\textbf{Proof of Proposition~\ref{prop:dual_net}}
From Proposition~\ref{prop:dual} it follows that $\Zm^{\circ\perp}$ is $(B^\top,A^\top)$-invariant and that $\Sm^{\circ\perp}$ is $(A^\top,C^\top)$-invariant. Suppose by contradiction that $\Zm^{\circ\perp}$ is not the minimal $(B^\top,A^\top)$-invariant node set containing $\Tm$. Then there must  exist an $(B^\top,A^\top)$-invariant node set $\Sm$ containing $\Tm$ s.t. $ | \Sm | < | \Zm^{\circ\perp} |$. Therefore, by Proposition~\ref{prop:dual}, $\exists$ an $(A,B)$-invariant node set $\Sm^\perp\subseteq\Vm\smallsetminus\Tm$ s.t. $ | \Zm^\circ | < | \Sm^\perp |$, which is a contradiction, since $\Zm^\circ$ is the maximal $(A,B)$-invariant node set contained in $\Vm\smallsetminus\Tm$. 
Similarly, suppose by contradiction that $\Sm^{\circ\perp}$ is not the maximal $(A^\top,C^\top)$-invariant node set contained in $\Vm\smallsetminus\Dm$. 
Then $\exists$ an $(A^\top,C^\top)$-invariant node set $\Zm$ contained in $\Vm\smallsetminus\Dm$ s.t. $ | \Zm | > | \Sm^{\circ\perp} |$. 
Hence, by Proposition~\ref{prop:dual}, $\exists$ a $(C,A)$-invariant node set $\Zm^\perp\supseteq\Dm$ s.t. $ | \Sm^\circ | > | \Zm^\perp |$, which is a contradiction, since $\Sm^\circ$ is the minimal $(C,A)$-invariant node set containing $\Dm$.
\qed

\textbf{Proof of Theorem~\ref{thm:relat}}
We note that proving items~$\ref{item:S_1}$ and $\ref{item:Z_1}$ of the Theorem is sufficient to prove the inclusions $\Sm^\circ\supseteq\Sm^\ast$ and $\Zm^\circ\subseteq\Zm^\ast$. We prove item~$\ref{item:S_1}$ and note that $\ref{item:Z_1}$ follows by duality, from Propositions~\ref{prop:dual} and \ref{prop:dual_net}. We say that $x\in\Sm^\ast$ is s.t. generically $x_i\neq 0$ iff there exists a basis of $\Sm^\ast$ containing at least one vector whose $i$-th component is nonzero. By construction of $\Sm^\circ$, node $v_i\notin\Sm^\circ$ iff $\nexists\,\Pm^{\Dm,v_i}\neq\emptyset$ or $\exists\,\Pm^{\Dm,v_i}\neq\emptyset$ s.t. $v_i\notin\partial_-(\Sm^\circ,A)$ and $\Pm^{\Dm,v_i}\cap\partial_-(\Sm^\circ,A)\neq\emptyset$ (note that this is a direct consequence of Lemma~\ref{lem:2} given below). Moreover, $\Dm\subseteq\Sm^\circ$ therefore $v_i\notin\Dm$. Note that by the construction in Proposition~\ref{prop:S} it follows $\partial_-(\Sm^\circ,A)\subseteq\Cm$. 
If $\Pm^{\Dm,v_i}=\emptyset$ then $\forall\,v_j\in\partial_-(\Sm^\circ,A)$, $x=A^\ell e_j$ has $x_i=0$, $\forall\ell\leq{n-1}$. If $\Pm^{\Dm,v_i}\neq\emptyset$ then $\forall\,v_j\in\partial_-(\Sm^\circ,A)$, $\exists\,\ell$ s.t. $x=A^\ell e_j$ has $x_i\neq0$. However, $e_j\in\im C^\top$ because of $\partial_-(\Sm^\circ,A)\subseteq\Cm$, thus from the recursion of Proposition~\ref{prop:cond_min_semi}, $e_j\notin\Sm_{\ell-1}\cap\ker C$ and such an $x$ cannot be in $\Sm^\ast$. Since the above considerations holds constructively for all $x\in\Sm^\ast$, we have proved that $v_i\notin\Sm^\circ\implies x_i=0,\,\forall\,x\in\Sm^\ast$, which corresponds to~$\ref{item:S_1}$.
\qed

\section{Proofs of Section~\ref{sec:standard-DDP}}
The following Lemmas are used to topologically characterize the DDP over networks.

\begin{lemma}\label{lem:1}
Consider the system~\eqref{eq:lin-syst2} and the recursion of Proposition~\ref{prop:Z}. Under Assumptions~\ref{ass:1} and \ref{ass:3}, if $\exists$ a path $\Pm^{v_i,v_t}=\{v_i,v_{i+1},\dots,v_t\}$ from $v_i\in\Vm$ to $v_t\in\Tm$ of length $\leq\ell$, and s.t. $\Pm^{v_{i+1},v_t}\cap\Vmin=\emptyset$, then $v_i\not\in\Zm_{m}$, $\forall\,m\geq\ell$.
\end{lemma}
\begin{proof}

Consider a generic path $\Pm^{v_i,v_t}$ from $v_i$ to $v_t\in\Tm$ of the form $\{v_i,\,v_{i+1},\dots,\,v_{t}\}$. At the first iteration of the algorithm of Proposition~\ref{prop:Z}, $\Zm_1=\Zm_0\smallsetminus\Km_1=(\Vm\smallsetminus\Tm)\smallsetminus\Km_1$, where $\Km_1$ contains all nodes directly connected to $\Tm\smallsetminus\Vmin$. 
Consequently, $\Km_1$ contains all the nodes at distance 1 from $\Tm\smallsetminus\Vmin$. At the second iteration, $\Zm_2=\Zm_1\smallsetminus\Km_2$, where $\Km_2$ contains all nodes directly connected to $\Km_1\smallsetminus\Vmin$. 
Consequently, a generic node $v_{k_2}\in\Km_2$ is at distance 2 from $\Tm\smallsetminus\Vmin$ and it is s.t. $\exists$ a sub-path $\Pm^{v_{k_2},v_t}$ of $ \Pm^{v_i,v_t}$, s.t. $(\Pm^{v_{k_2},v_t}\smallsetminus\{v_{k_2}\})\cap\Vmin=\emptyset$.
Iterating, $\Zm_p=\Zm_{p-1}\smallsetminus\Km_{p}$, where $\Km_{p}$ contains all nodes directly connected to $\Km_{p-1}\smallsetminus\Vmin$.
Consequently, a generic node $v_{k_{p}}\in\Km_{p}$ is at distance $p$ from $\Tm\smallsetminus\Vmin$ and it is s.t. $\exists$ a sub-path $\Pm^{v_{k_{p}},v_t}$ of $ \Pm^{v_i,v_t}$, s.t. $(\Pm^{v_{k_{p}},v_t}\smallsetminus\{v_{k_{p}}\})\cap\Vmin=\emptyset$.
Let us assume that $(\Pm^{v_i,v_t}\smallsetminus\{v_t\})\cap\Tm=\emptyset$, where the length of $\Pm^{v_i,v_t}$ is $\ell$. 
By the previous construction, we note that if $\Pm^{v_{i+1},v_t}\cap\Vmin=\emptyset$, then $v_i\not\in\Zm_{m}$, $\forall\,m\geq\ell$. 
If instead $\exists\,v_k\in(\Pm^{v_i,v_t}\smallsetminus\{v_t\})\cap\Tm$, $k\in\{i,\dots,\,t-1\}$, we consider the sub-path $\Pm^{v_i,v_k}$, of length $\ell-(t-k)$. 
In this case, since $k<t$ we have that $v_i\notin\Zm_r$, $\forall\,r\geq\ell-(t-k)$. Moreover, if $\exists\,v_q\in(\Pm^{v_i,v_t}\smallsetminus\{v_t,\,v_k\})\cap\Tm$, two cases may arise. If $q\in\{i,\dots,\,k-1\}$, since $q<k$, we have that $v_i\notin\Zm_f$, $\forall\,f\geq\ell-(t-q)$. Otherwise, if $q\in\{k+1,\dots,\,t-1\}$, since $k<q$, the considerations of the case in which only the target node $v_k$ appears in $\Pm^{v_i,v_t}\smallsetminus\{v_t\}$ can be applied, i.e., $v_i\notin\Zm_r$, $\forall\,r\geq\ell-(t-k)$. Thus, by iteration, if $\exists$ a path $\Pm^{v_i,v_t}$ of length $\leq\ell$ and s.t. $\Pm^{v_{i+1},v_t}\cap\Vmin=\emptyset$, then $v_i\not\in\Zm_{m}$, $\forall\,m\geq\ell$.
\end{proof}

\begin{lemma}\label{lem:2}

Consider the system~\eqref{eq:lin-syst2} and the recursion of Proposition~\ref{prop:S}. Under Assumptions~\ref{ass:2} and \ref{ass:3}, if $\exists$ a path $\Pm^{v_1,v_i}=\{v_1,\dots,v_{i-1},v_i\}$ from $v_1\in\Dm$ to $v_i\in\Vm$ of length $\leq\ell$, and s.t. $\Pm^{v_{1},v_{i-1}}\cap\Vmout=\emptyset$, then $v_i\in\Sm_{m}$, $\forall\,m\geq\ell$.
\end{lemma}
\begin{proof}
The proof follows a reasoning similar to that in Lemma~\ref{lem:1} and is therefore omitted.
\end{proof}

\textbf{Proof of Theorem~\ref{thm:ddpsf}}
From Theorem~\ref{thm:relat}, $\Dm\subseteq\Zm^\circ(\Vmin)\implies\Dm\subseteq\Zm^\ast$, since $\Zm^\circ\subseteq\Zm^\ast$, hence the DDPSF is solvable when $\ref{eq:DDPSF-iff}$ holds.\\
$(\ref{eq:DDPSF-iff}~\implies\ref{item:c_thm:ddpsf})$:
To prove the \emph{necessity}, we use a contradiction argument. Let us assume that $\exists\,j\,\,\text{s.t.}\,\,\Vmin\cap \mathcal{P}_j=\emptyset$ and, w.l.o.g., let us relabel the network nodes so that $\mathcal{P}_j=\{v_{1},\,v_{2},\dots,\,v_{{k-1}},\,v_{k}\}$, with $v_{1}\in \mathcal{D}$, $v_{k}\in \mathcal{T}$, and with path length $k-1\geq2$. Let us consider $\Zm_0=\Vm\smallsetminus\Tm$ and let us add the assumption that $\{v_{2},\dots,\,v_{{k-1}}\}\cap\Tm=\emptyset$ so that $\{v_{1},\dots,\,v_{{k-1}}\}\subseteq\Zm_0$. Consider the first iteration of the recursion of Proposition~\ref{prop:Z}, for which $\Zm_1=\Zm_0 \smallsetminus \{ v_i \in \Zm_0\,\,\text{s.t.}\,\,(v_i, \, v_j ) \in \Em\,\,\text{and}\,\,v_j \in \Vm\smallsetminus(\Zm_0 \cup \Vmin) \}$.
Since $v_{k}\notin\Zm_0$ and, by the contradictory assumption, $v_{k}\notin\Vmin$, we note that $v_{k}\in\Vm\smallsetminus(\Zm_0 \cup \Vmin)$. Moreover, from the existence of $\Pm_j$ we have $(v_{{k-1}},\,v_{k})\in\Em$, which implies that $v_{{k-1}}\notin\Zm_1$. By iterating this argument, we can remove from $\Zm_q$, $q\in\{2,\dots,\,k-1\}$, at least the node $v_{{k-q}}$. Therefore, when $q=k-1$, the node $v_{1}$ will be removed from $\Zm_{k-1}$, implying that $\Dm\not\subseteq\Zm_{k-1}\implies\Dm\not\subseteq\Zm^\circ$, since $\Zm^\circ\subseteq\Zm_{k-1}$. This is a contradiction, as we have assumed that $\Dm\subseteq\Zm^\circ(\Vmin)$ holds.
This last result can be easily generalized to the case in which $\{v_{2},\dots,\,v_{k-1}\}\cap\Tm\neq\emptyset$. If w.l.o.g., we consider $s^\ast=\min\{s\in\{2,\dots,\,k-1\}\,\,\text{s.t.}\,\,v_{s}\in\Tm\cap\Pm_j\}$,
the above result can then be applied to the sub-path $\{v_{1},\,v_{2},\dots,\,v_{{s^\ast-1}},\,v_{s^\ast}\}$.\\
$(\ref{item:c_thm:ddpsf}\implies~\ref{eq:DDPSF-iff})$: To prove the \emph{sufficiency}, consider a $ \Dm$-to-$\Tm$ path $\mathcal{P}_j$, and w.l.o.g. relabel the networks nodes so that $\mathcal{P}_j=\{v_{1},\,v_{2},\dots,\,v_{{k-1}},\,v_{k}\}$, with $v_{1}\in \mathcal{D}$ and $v_{k}\in \mathcal{T}$. 
By assumption, there exists at least a node $v_i\in\Vmin\cap\Pm_j$. Since $\Vmin\cap\Dm=\emptyset$, if we consider $i=\max\{s\in\{2,\dots,\,k\}\,\,\text{s.t.}\,\,v_{s}\in\Vmin\cap\Pm_j\}$, from Lemma~\ref{lem:1} all nodes in the sub-path $\mathcal{P}_j\smallsetminus\{v_{i},\dots,\,v_{k}\}$ will not be removed from $\Zm_{\hat{k}}$ for some $\hat{k}\geq k-i$, 
unless there exists another $ \Dm$-to-$\Tm$ path $ \Pm_h =\{ v_\alpha, \ldots , v_\omega \} $ intersecting $ \Pm_j \smallsetminus\{v_i,\dots,v_k\}$ in at least one node $ v_\lambda $ and for the sub-path $\Pm_h^{v_\lambda,v_\omega}$ of $ \Pm_h  $ it is $\Pm_h^{v_{\lambda+1},v_\omega} \cap\Vmin=\emptyset$. But then there must be at least another node $ v_\ell $ in the sub-path $ \Pm_j^{v_1, v_\lambda}$ of $  \Pm_j $ s.t. $ v_\ell \in \Vmin$, since otherwise the composite $ \Dm$-to-$\Tm$ path $ \Pm_j^{v_1, v_\lambda} \cup \Pm_h^{v_{\lambda+1},v_\omega} $ has no intersection with $ \Vmin$, which contradicts the assumption of the theorem. 
The argument can be repeated $\forall$ $ \Dm$-to-$\Tm$ paths intersecting $ \Pm_j$.
There must therefore exist an index $ r\leq |\Pm_j|-1 = k-1 $ s.t. $ \Pm_j^{v_1, v_{k-r}} \subseteq\Zm_q=\Zm^\circ$, $\forall\,q\geq r$, meaning that $ v_1  \in \Dm $ is also $  v_1 \in \Zm^\circ$.
Since this must be valid for each $ \Dm$-to-$\Tm$ path $ \Pm_j$, it must be $\Dm\subseteq\Zm_q=\Zm^\circ$, $\forall\,q\geq \hat{r}$ for some $ \hat{r}$.
\qed

The conditions of Theorem~\ref{thm:ddpof} correspond to a special case of those of Theorem~\ref{thm:ddpdf}, therefore the proof of Theorem~\ref{thm:ddpof} is postponed after that of Theorem~\ref{thm:ddpdf}.

\textbf{Proof of Theorem~\ref{thm:ddpdf}}
From Theorem~\ref{thm:relat}, $\Sm^\circ(\Cm)\subseteq\Zm^\circ(\Vmin)\implies\Sm^\ast\subseteq\Zm^\ast$, since $\Zm^\circ\subseteq\Zm^\ast$ and $\Sm^\ast\subseteq\Sm^\circ$, hence $\ref{item:c_thm:ddpdf}$ implies the solvability of the DDPDF.\\
$(\ref{item:d_thm:ddpdf}\implies\ref{item:c_thm:ddpdf})$:
We use a contradiction argument. By construction $\Sm_0\subseteq\Zm_0$. From Propositions~\ref{prop:Z}~and~\ref{prop:S}, $\Sm_k$ grows by adding nodes to $\Sm_{o}$ that belong to paths that originate in $\Dm\smallsetminus\Vmout$, and $\Zm_k$ shrinks by removing nodes from $\Zm_0$ that belong to paths that end in $\Tm\smallsetminus\Vmin$. Thus, since a path originating in $\Dm\smallsetminus\Vmout$ not ending in $\Tm$ can be eventually\footnote{By Lemma~\ref{lem:2}, if such path does not contain output nodes.} added to $\Sm_0$ but will be contained in $\Zm_k$ $\forall\,k$, and since a path ending in $\Tm\smallsetminus\Vmin$ not starting from $\Dm$ can be eventually\footnote{By Lemma~\ref{lem:1}, if such path does not contain input nodes.} removed from $\Zm_0$, but will not be part of $\Sm_k$ $\forall\,k$, the only way $\Sm_\lambda\not\subseteq\Zm_\lambda$ for some $\lambda$ is when at least a node involved in a $\Dm$-to-$\Tm$ path is added in $\Sm_p$ but removed from $\Zm_q$, for some $p,q\leq \lambda$. 
Consider a generic $\Dm$-to-$\Tm$ path $\Pm_j$ of length $\kappa_j-1\geq1$ and, w.l.o.g., relabel the network nodes so that $\Pm_j=\{v_1,\dots,v_{\kappa_j}\}$, with $v_1\in\Dm$ and $v_{\kappa_j}\in\Tm$ and s.t. $o_j<i_j$ (note that the existence of such indexes $\forall\,j\in\{1,\dots,\ell\}$ implicitly implies that every $\Dm$-to-$\Tm$ path contains at least one control node and one output node). 
Thus, by Lemma~\ref{lem:2}, $\Pm_j^{v_1,v_{o_j}}\subseteq\Sm_m$, $\forall\,m\geq o_j-1$. Similarly, by Lemma~\ref{lem:1}, $\Pm_j^{v_{i_j},v_{\kappa_j}}\cap\Zm_d=\emptyset$, $\forall\,d\geq \kappa_j-i_j$. 
A generic node $v_{s}\in\Pm_j^{v_1,v_{i_j-1}}$ cannot be removed from $\Zm_q$, after some $p$, unless $v_{s}\in\Pm_j^{v_{o_j},v_{i_j-1}}$ and some other $ \Dm$-to-$\Tm$ path intersects $ \Pm_j$. 
Let us assume by contradiction that $v_{s}\in\Pm_j^{v_1,v_{o_j}}$ is removed from $\Zm_q$, after some $p\geq1$, which means that there must $\exists$ another $ \Dm$-to-$\Tm$ path $ \Pm_h =\{ v_\alpha, \ldots , v_\omega \} $ s.t. $v_s\in\Pm_h^{v_\alpha,v_{\omega}}$ and $\Pm_h^{v_{s+1},v_{\omega}}\cap\Vmin=\emptyset$. However, by assumption, for such path there exist indexes $ o_h $ and $ i_h$ satisfying $o_h<i_h$, which implies that there must exist a control node $v_{i_h}\in\Pm_h^{v_\alpha,v_s}$. Therefore, the $ \Dm$-to-$\Tm$ path $\Pm_r=\Pm_j^{v_1,v_{s}}\cup\Pm_h^{v_{s+1},v_\omega}$ must have $v_{i_r}\in\Pm_j^{v_1,v_{s}}$ and since condition $o_r<i_r$ must hold by assumption for such path, it follows that $o_r<o_j$, which is a contradiction, since $o_j$ is by definition the minimal output node index on the path $\Pm_j$. Let us now consider the case in which $v_{s}\in\Pm_j^{v_{o_j},v_{i_j-1}}$. 
For such node to be removed from $\Zm_q$, after some $p$, there must $\exists$ another $ \Dm$-to-$\Tm$ path $ \Pm_h =\{ v_\alpha, \ldots , v_\omega \} $ s.t. $v_s\in\Pm_h^{v_\alpha,v_{\omega}}$ and $\Pm_h^{v_{s+1},v_{\omega}}\cap\Vmin=\emptyset$. We show that node $v_s$ cannot be added to $\Sm_p$, after some $p$. In fact, let us assume by contradiction that $v_s\in\Sm_p$, after some $p$, which means that there $\exists$ another $ \Dm$-to-$\Tm$ path $ \Pm_f =\{ v_\gamma, \ldots , v_\delta \} $ s.t. $v_s\in\Pm_f^{v_\gamma,v_{\delta}}$ and $\Pm_f^{v_\gamma,v_{s-1}}\cap\Vmout=\emptyset$. However, the $ \Dm$-to-$\Tm$ path $\Pm_z=\Pm_f^{v_\gamma,v_{s-1}}\cup\{v_s\}\cup\Pm_h^{v_{s+1},v_\omega}$, by construction, does not satisfy condition $o_z<i_z$, which is a contradiction. This reasoning can be applied to each $ \Dm$-to-$\Tm$ path $ \Pm_j$, $j\in\{1,\dots,\ell\}$. Therefore, when $o_j<i_j$, $\forall\,j\in\{1,\dots,\ell\}$, there exists an index $k_o=\max_j(o_j-1)$ s.t. $\Sm_p=\Sm^\circ$, $\forall\,p\geq k_o$, and an index $k_i=\max_j(\kappa_j-i_j)$ s.t. $\Zm_q=\Zm^\circ$, $\forall\,q\geq k_i$, for which $\forall\lambda\geq\max\{k_o,k_i\}$ it holds $\Sm^\circ=\Sm_\lambda\subseteq\Zm_\lambda=\Zm^\circ$.\\
$(\ref{item:c_thm:ddpdf}\implies\ref{item:d_thm:ddpdf})$:
Assume by contradiction that $\exists$ a $\Dm$-to-$\Tm$ path $\Pm_j=\{v_1,\dots,v_{\kappa_j}\}$, with $v_1\in\Dm$ and $v_{\kappa_j}\in\Tm$ s.t. $o_j\geq i_j$. By Lemma~\ref{lem:2}, $\Pm_j^{v_1,v_{o_j}}\subseteq\Sm_m$, $\forall\,m\geq o_j-1$, and by Lemma~\ref{lem:1}, $\Pm_j^{v_{i_j},v_{\kappa_j}}\cap\Zm_d=\emptyset$, $\forall\,d\geq \kappa_j-i_j$. Since $v_{i_j}\in\Pm_j^{v_1,v_{o_j}}$ and $v_{o_j}\in\Pm_j^{v_{i_j},v_{\kappa_j}}$, it is guaranteed that $\Sm_\lambda\not\subseteq\Zm_\lambda$, $\forall\,\lambda\geq\min\{\max\{i_j-1,\,\kappa_j-i_j\},\,\max\{o_j-1,\,\kappa_j-o_j\}\}$, which contradicts item $\ref{item:c_thm:ddpdf}$.
\qed

To prove Theorem~\ref{thm:ddpof} we need the following Lemma.
\begin{lemma}
\label{lemma:cond-contr-inv}
Under Assumptions~\ref{ass:1}--\ref{ass:3}, if $\Wm$ is a controlled and conditioned invariant set of nodes, then $\partial_-(\Wm,A)\subseteq\Vmout$ and $\partial_+(\Wm,A) \subseteq\Vmin$.
\end{lemma}
\begin{proof}
According to Definition~\ref{def:boundary}, each node in the in-boundary $\partial_-(\Wm,A)\subseteq \Wm$ has an edge landing in the out-boundary $\partial_+(\Wm,A)\subseteq \Wm^\perp$. Since $ \Wm $ is conditioned invariant, the starting node must necessarily be in $ \Vmout $ and, since $ \Wm $ is controlled invariant, the landing node must necessarily be in $ \Vmin$. 
\end{proof}

\textbf{Proof of Theorem~\ref{thm:ddpof}}
From Definition~\ref{def:sub}, a controlled and conditioned invariant node set $\Wm$ is also a controlled and conditioned invariant subspace.\\
$(\ref{item:a_thm:ddpof}\implies\ref{item:c_thm:ddpof})$: A set of nodes $\Wm$ is $(A,B)$-invariant and $(C,A)$-invariant iff it is $A$-invariant or each edge exiting $\Wm$ is of the type $(v_i,v_j)$ with $v_i\in\Vmout$ and $v_j\in\Vmin$. In fact, if by contradiction $\exists$ an edge exiting $\Wm$ not of the type $(v_i,v_j)$ with $v_i\in\Vmout$ and $v_j\in\Vmin$, either condition~$\ref{item:a_prop:AB-invar-equiv-char}$ of Proposition~\ref{prop:AB-invar-equiv-char}, or condition~$\ref{item:a_prop:CA-invar-equiv-char}$ of Proposition~\ref{prop:CA-invar-equiv-char}, or both, are violated. 
Since $\Wm$ is controlled and conditioned invariant, from Lemma~\ref{lemma:cond-contr-inv} it must be $\partial_-(\Wm,A)\subseteq\Vmout$ and $\partial_+(\Wm,A)\subseteq\Vmin$. W.l.o.g., it is always possible to relabel the network nodes s.t. a generic $\Dm$-to-$\Tm$ path can be written as in $\ref{item:c_thm:ddpof}$. 
Since $\Dm\subseteq\Wm\subseteq\Vm\smallsetminus\Tm$, a generic $\Dm$-to-$\Tm$ path must contain an edge from $\partial_-(\Wm,A)$ to $ \partial_+(\Wm,A)$), so there must exist a sub-path of length 1 of the form $\{v_{p},v_{{p+1}}\}$ with $v_{p}\in\Vmout$ and $v_{p+1}\in\Vmin$.\\
$(\ref{item:c_thm:ddpof}\implies\ref{item:a_thm:ddpof})$: Let $\Pm_1,\dots,\Pm_\ell$ be all possible $\Dm$-to-$\Tm$ paths. 
Let us consider a generic $\Dm$-to-$\Tm$ path $\Pm_j$, of length $\kappa_j-1\geq1$. By hypothesis, $\exists$ a sub-path of length 1 of the form $\{v_{p},v_{p+1}\}$ with $v_{p}\in\Vmout$ and $v_{p+1}\in\Vmin$, $p\in\{1,\dots,\kappa_j-1\}$, $j\in\{1,\dots,\ell\}$. 
It follows that $o_j\leq p$ and $i_j\geq p+1$, where $o_j$ and $i_j$ are the indexes defined in item $\ref{item:d_thm:ddpdf}$ of Theorem~\ref{thm:ddpdf}, which implies that $o_j<i_j$. 
Since, by assumption, this holds $\forall\,j\in\{1,\dots,\ell\}$, then $\Sm^\circ(\Vmout)\subseteq\Zm^\circ(\Vmin)$, by Theorem~\ref{thm:ddpdf}. Condition $\Sm^\circ(\Vmout)\subseteq\Zm^\circ(\Vmin)$ alone does not guarantee the existence of a set $\Wm$ both controlled and conditioned invariant. In fact, given $v_i\in\Vmout$, $v_j\in\Vmin$ suppose there exists a path $\Pm^{v_i,v_j}=\{v_i,v_k,v_j\}$ with $v_k\in\Vm\smallsetminus(\Vmout\cup\Vmin)$. Therefore, if we assume $\exists$ a controlled invariant superset $\Wm$ of $\Sm^\circ$ this means that it must contain node $v_k$. However, $v_k\notin\Vmout$ implies that $\Wm$ is not conditional invariant, i.e., the edge $(v_k,v_j)$ exits $\Wm$ and is not of the type $(v_i,v_j)$ with $v_i\in\Vmout$ and $v_j\in\Vmin$. Similarly, if we assume $\exists$ a conditioned invariant subset $\Wm$ of $\Zm^\circ$ this means that it must exclude node $v_k$. 
However, $v_k\notin\Vmin$ implies that $\Wm$ is not controlled invariant, i.e., the edge $(v_i,v_k)$ exits $\Wm$ and is not of the type $(v_i,v_j)$ with $v_i\in\Vmout$ and $v_j\in\Vmin$. These arguments can be extended to longer paths. Let us consider the partitions $\Vmout=\Vmout^{\Sm^\circ}\cup\Vmout^{\Sm^{\circ\perp}}$ and $\Vmin=\Vmin^{\Zm^{\circ\perp}}\cup\Vmin^{\Zm^{\circ}}$, where $\Vmout^{\Sm^\circ}\subseteq\Sm^\circ$, $\Vmout^{\Sm^{\circ\perp}}\subseteq\Sm^{\circ\perp}$, $\Vmin^{\Zm^{\circ\perp}}\subseteq\Zm^{\circ\perp}$, and $\Vmin^{\Zm^{\circ}}\subseteq\Zm^\circ$. 
By Proposition~\ref{prop:S}, if $\exists$ an edge $(v_i,v_j)$ exiting $\Sm^\circ$ then $v_i\in\Vmout^{\Sm^\circ}$. If $\Sm^\circ$ is not controlled invariant, it means that $\exists$ at least an edge $(v_i,v_j)$ exiting $\Sm^\circ$ s.t. $v_j\notin\Vmin$. 
Condition $\Sm^\circ\subseteq\Zm^\circ$ ensures that $v_j\in\Zm^\circ\smallsetminus\Sm^\circ$. 
In fact, by Proposition~\ref{prop:Z}, the only edges allowed to enter $\Zm^{\circ\perp}$ must land on $\Vmin$, thus the only condition in which there may exist an edge $(v_i,v_j)$ exiting $\Sm^\circ$ with $v_i\in\Sm^\circ$ but $v_j\notin\Zm^\circ$ is when $v_j\in\Vmin^{\Zm^{\circ\perp}}$, which implies controlled invariance of $ \Sm^\circ$ and the proof is finished. 
If instead $ v_j \in \Zm^\circ$, then let us construct a growing sequence of conditioned invariant subsets which we denote $ \Sm^{k,\circ}$ (with $ \Sm^{1,\circ}= \Sm^\circ$).
If we consider the out-boundary of the set $\Sm^\circ$, defined as $\Fm_1=\Sm^{\circ\perp}\cap\mathbb{A}\Vmout^{\Sm^\circ}$,
with $\mathbb{A}\Vmout^{\Sm^\circ}$ denoting the set of children of $\Vmout^{\Sm^\circ}=\Vmout\cap\Sm^\circ$, and the initialization $\Sm_0^2=\Sm^\circ\cup(\Fm_1\smallsetminus\Vmin)$, then we ensure that $v_j\in\Sm_0^2$. 
By applying the recursion of Proposition~\ref{prop:S}, we converge to the minimal conditioned invariant set $\Sm^{2,\circ}$ containing $\Sm_0^2$ (and thus $\Sm^\circ$ and $\Dm$).
If $\Sm^{2,\circ}$ is not controlled invariant, it means that $\exists$ at least an edge $(v_i,v_j)$ exiting $\Sm^{2,\circ}$ s.t. $v_j\notin\Vmin$ and $v_i\in\Vmout^{\Sm^{2,\circ}}\subseteq\Vmout^{\Sm^{\circ\perp}}$, with $\Vmout^{\Sm^{2,\circ}}=\Vmout^{\Sm^{\circ\perp}}\cap\Sm^{2,\circ}$ denoting the set of output nodes in $\Sm^{2,\circ}\smallsetminus\Sm^\circ$ \footnote{Given three sets $A$, $B$ and $C$, and the inclusion $A\subseteq B\subseteq C$, it is $B\smallsetminus A=(C\smallsetminus A)\cap B$. Therefore, by the associativity of the intersection and since $\Sm^\circ\subseteq\Sm^{2,\circ}\subseteq\Vm$, it follows that $\Vmout^{\Sm^{2,\circ}}=\Vmout^{\Sm^{\circ\perp}}\cap\Sm^{2,\circ}=(\Vmout\cap\Sm^{\circ\perp})\cap\Sm^{2,\circ}=(\Vmout\cap(\Vm\smallsetminus\Sm^\circ))\cap\Sm^{2,\circ}=\Vmout\cap((\Vm\smallsetminus\Sm^\circ)\cap\Sm^{2,\circ})=\Vmout\cap(\Sm^{2,\circ}\smallsetminus\Sm^\circ)$.}. 
If we consider the out-boundary of the set $\Sm^{2,\circ}\smallsetminus\Sm^\circ$, $\Fm_2=\Sm^{2,\circ\perp}\cap\mathbb{A}\Vmout^{\Sm^{2,\circ}}$,
a new initialization $\Sm_0^3=\Sm^{2,\circ}\cup(\Fm_2\smallsetminus\Vmin)$ guarantees convergence to another conditional invariant set ${\Sm^{3,\circ}}$ containing $\Sm_0^3$ (and thus $\Sm^{2,\circ}$, $\Sm^\circ$ and $\Dm$). 
At iteration $k>1$, the out-boundary of the set $\Sm^{k,\circ}\smallsetminus\Sm^{(k-1),\circ}$ becomes $\Fm_k=\Sm^{k,\circ\perp}\cap\mathbb{A}\Vmout^{\Sm^{k,\circ}}$,
with $\Vmout^{\Sm^{k,\circ}}=\Vmout^{\Sm^{(k-1),\circ\perp}}\cap\Sm^{k,\circ}$. 
Therefore, by applying the recursion of Proposition~\ref{prop:S} with the initialization $\Sm_0^{k+1}=\Sm^{k,\circ}\cup(\Fm_k\smallsetminus\Vmin)$, we converge to the minimal conditioned invariant set $\Sm^{(k+1),\circ}$ containing $\Sm_0^{k+1}\supseteq\Dm$. The iterative scheme converges when $\Sm^{(k+1),\circ}=\Sm_0^{k+1}=\Sm^{k,\circ}$, i.e., when $\Fm_k\smallsetminus\Vmin=\emptyset$ or, equivalently, when $\Fm_k\subseteq\Vmin$. 
Since $\Fm_k$ denotes the out-boundary of the set $\Sm^{k,\circ}\smallsetminus\Sm^{(k-1),\circ}$, this implies control invariance for $\Sm^{k,\circ}$. Finally, since by hypothesis $\exists$ at least one sub-path of length 1 of the form $\{v_{p},v_{p+1}\}$ with $v_{p}\in\Vmout$ and $v_{p+1}\in\Vmin$ for each $\Dm$-to-$\Tm$ path, the conditional and controlled invariant set $\Sm^{k,\circ}$ is s.t. $\Sm^{k,\circ}\subseteq\Zm^\circ\subseteq\Vm\smallsetminus\Tm$. We can set $\Wm=\Sm^{k,\circ}$, which concludes the proof.
The number of iterations before convergence of the proposed scheme depends on the distribution of the output nodes in the paths and it is upper bounded by $|\Vmout^{\Sm^{\circ\perp}}|+1$. Therefore, $\exists\,s\leq |\Vmout^{\Sm^{\circ\perp}}|+1$ s.t. $\Sm^{k,\circ}=\Sm^{(k+1),\circ}$, $\forall\;k\geq s$, and it is both controlled and conditioned invariant. The proof can be dualized on $\Zm^\circ$, by showing that under condition $\ref{item:c_thm:ddpof}$, it is always possible to converge to a subset of $\Zm^\circ$ that is both controlled invariant and conditioned invariant.
\qed

\textbf{Proof of Corollary~\ref{cor:ddpof}}
It is sufficient to observe that, since $\Sm^\circ(\Vmout)$ is the minimal conditional invariant node set containing $\Dm$ it follows that necessarily $\Sm^\circ\subseteq\Wm$. Moreover, since $\Zm^\circ(\Vmout)$ is the maximal control invariant node set contained in $\Vm\smallsetminus\Tm$, $\Wm\subseteq\Zm^\circ$ necessarily. 
\qed

\section{Proofs of Section~\ref{sec:min-node-DDP}}
\textbf{Proof of Proposition~\ref{prop-cut-extended}}
In $ \overline{\Gm}$ all ``true'' edges of $ \Gm $ have $ \infty $ weight, hence they are never selected by the min-cut algorithm.
Therefore all edges on which each cut set is computed in $ \overline{\Gm}$ must have the form $ (v_i, v_{i+n})$, where $ v_i $ and $ v_{i+n} $ are duplicated nodes, and these edge weights are all equal to 1. By construction, each cut has value equal to both the number of edges across it and hence to the number of nodes in the corresponding node cut set in~$\Gm$. This implies that the min-cut value is equal to $ | \Vmin|$.
\qed

\textbf{Proof of Theorem~\ref{thm:main_paper2}}
Each solution obtained through Algorithms~\ref{alg:one}--\ref{alg:three} defines an edge cut set on the extended network $\overline{\mathcal{G}}$ which, from Proposition~\ref{prop-cut-extended}, corresponds to a node cut set on $ \Gm$, denoted $ \Vmin$. By construction, Algorithms~\ref{alg:one}--\ref{alg:three} ensure that $\Vmin \cap \Pm_j \neq \emptyset$ $\forall$ $\Dm$-to-$\Tm$ path $\Pm_j$ of $ \Gm$, which guarantees the solvability of the DDPSF by item $\ref{item:c_thm:ddpsf}$ of Theorem~\ref{thm:ddpsf}.
Optimality arises from the minimality of the cut set's cardinality. 
By contradiction assume that $\exists$ $ \Vmin^\prime $ s.t. $ | \Vmin^\prime |< | \Vmin|$ that solves Problem~2A. 
Then, from Proposition~\ref{prop-cut-extended}, $ \Vmin^\prime $ must produce a node cut set in $ \Gm $ of smaller cardinality than the one associated with $ \Vmin$, meaning, again from Proposition~\ref{prop-cut-extended}, that also the edge cut set in $ \overline{\Gm}$ must be smaller than the one associated with $ \Vmin$. This however means that the min-cut/max-flow algorithm is not optimal, which is a contradiction. 
\qed

\textbf{Proof of Theorem~\ref{thm:SF_otp_charact}}
$\ref{item:a_SF_otp_charact}$: If $\exists\,v_i\in\partial_+(\Zm^\circ(\Vmin),A)\smallsetminus\Vmin$, then $\Zm^\circ(\Vmin)$ is not controlled invariant.\\
$\ref{item:b_SF_otp_charact}$: Saying that $\Vmin$ identifies a cut set of nodes for $\Gm$ means that the set of edges $\Hm=\{(v_i,\,v_{i+n})\in\overline{\Em}\,\,\text{s.t.}\,\,v_i\in\Vmin\}$ is a cut set of edges for $\overline{\Gm}$, i.e., when removed, there are no more direct paths connecting $v^\Dm$ to $v^\Tm$. If we call $\hat{\Bm}=\{v_{i+n}\,\,\text{s.t.}\,\,v_i\in\Vmin\,\,\text{and}\,\,(v_i,\,v_{i+n})\in\Hm\}$, then $\Hm=\Vmin\times\hat{\Bm}$. 
Let \(\overline{\Jm}^\Dm \subseteq \overline{\Vm} \) and \( \overline{\Jm}^\Tm = \overline{\Vm} \smallsetminus \overline{\Jm}^\Dm \) be the partition of \( \overline{\Vm} \) induced by the cut.
As a consequence, $\Vmin\subseteq\overline{\Jm}^\Dm$ and $\hat{\Bm}\subseteq\overline{\Jm}^\Tm$. Denote by $\overline{\Em}^{\Dm}$ and $\overline{\Em}^{\Tm}$ the sets of edges within $\overline{\Jm}^\Dm$ and $\overline{\Jm}^\Tm$, respectively. The set of edges with tail in $\overline{\Jm}^\Dm$ and head in $\overline{\Jm}^\Tm$ is denoted by $\overline{\Em}^{\Dm\to\Tm}$, while the set of edges with tail in $\overline{\Jm}^\Tm$ and head in $\overline{\Jm}^\Dm$ is denoted by $\overline{\Em}^{\Tm\to\Dm}$. By construction, it follows $\overline{\Em}^{\Dm\to\Tm}=\Hm$. In fact, if by contradiction a node in $\overline{\Jm}^\Tm\smallsetminus\hat{\Bm}$ is reached by an edge with tail in $\overline{\Jm}^\Dm\smallsetminus\Vmin$, then it is also reachable from $v^\Dm$, which is a contradiction according to the definition of $ \overline{\Jm}^\Dm $ given above. 
By construction we write the extended network as $\overline{\Gm}=(\overline{\Vm},\overline{\Em}, \overline{A})=(\overline{\Jm}^\Dm\cup\overline{\Jm}^\Tm,\overline{\Em}^{\Dm}\cup\overline{\Em}^{\Tm}\cup\overline{\Em}^{\Tm\to\Dm}\cup\Hm, \overline{A})$.
Consider a generic path $\overline{\Pm}^{v^\Dm,v^\Tm}$ in $ \overline{\Gm}$ from $v^\Dm$ to $v^\Tm$. Assume it has length $\geq 7$ and that it intersects $\Dm$, $\Tm$ and $\Vmin$ only once\footnote{In the case of length 5 a duplicate disturbance node $v_{d+n}$, with $v_{d}\in\Dm$, is directly connected to a target node $v_{t}\in\Tm$, i.e., $\exists$ a path $\overline{\Pm}^{v^\Dm,v^\Tm}=\{v^\Dm,\,v_{d},\,v_{d+n},\,v_{t},\,v_{t+n},\,v^\Tm\}$. 
Algorithms~\ref{alg:one}--\ref{alg:three} are s.t. $v_{t}\in\Vmin$, thus, since $\overline{\Pm}^{v_{t+n},v^\Tm}\cap\Vmin=\emptyset$, by Lemma~\ref{lem:1} $v_{t}\notin\Zm^\circ$. 
}. Such path must exist because of the decomposition of $\overline{\Gm}$ induced by the cut $\Hm$. Moreover, $\exists\,v_i\in\Vmin$, $v_{i+n}\in\hat{\Bm}$ s.t. $\overline{\Pm}^{v^\Dm,v^\Tm}=\{v^\Dm,\,v_{d},\,v_{d+n},\dots,\,v_i,\,v_{i+n},\dots,\,v_{t},\,v_{t+n},\,v^\Tm\}$. Assume that $\overline{\Pm}^{v^\Dm,v^\Tm}\cap\Vmin=v_{i}$ is a single node, but one can relax this assumption and apply Lemma~\ref{lem:1} on the shortest sub-path $\overline{\Pm}^{v_{i},v^\Tm}$ of $\overline{\Pm}^{v^\Dm,v^\Tm}$, with $v_{i}\in\Vmin$ and $\overline{\Pm}^{v_{i},v^\Tm}\cap\Vmin=v_{i}$ only. After relabeling the network nodes, Lemma~\ref{lem:1} guarantees that $\exists\,k$ s.t. $v_i \notin \Zm_k$, where $\Zm_k$ is the $k$-th iteration of the recursion of Proposition~\ref{prop:Z}. 
Hence $v_i\notin\Zm^\circ$ and $ v_i \in \partial_+(\Zm^\circ(\Vmin),A)$. Iterating the previous reasoning on all the paths connecting $v^\Dm$ to $v^\Tm$, due to the structure of the cut set $\Hm$,
we obtain $\Vmin=\partial_+(\Zm^\circ(\Vmin),A)$. Finally, we note that necessarily $\partial_+(\Zm^\circ(\Vmin),A) = \partial_+(\Zm^\circ(\Vmin),A) \cap \mathcal{P}$, i.e., $\Vmin\subseteq\Pm$, otherwise $\Bm $ cannot be minimal (and hence optimal).\\
$\ref{item:c_SF_otp_charact}$: This is straightforward given the construction of cut set of nodes given above.
\qed

\textbf{Proof of Proposition~\ref{thm:sub}}
Since $\Vmin$ is a feasible solution of Problem~2A it follows from Theorem~\ref{thm:SF_otp_charact}$(a)$ that $\partial_+(\Zm^\circ(\Vmin),A)\subseteq\Vmin$. From controlled invariance, $\forall\,x\in\Zm^\circ$ we have $Ax\in\Zm^\circ+\partial_+(\Zm^\circ(\Vmin),A)$, therefore $\Zm^\circ$ is $(A,\Tilde{B})$-invariant for each submatrix $\Tilde{B}$ of $B$ whose range contains $\partial_+(\Zm^\circ(\Vmin),A)$.
\qed

\textbf{Proof of Theorem~\ref{thm:OF_otp_charact}}
$\ref{item:a_OF_otp_charact}$: See Lemma~\ref{lemma:cond-contr-inv}.\\
$\ref{item:b_OF_otp_charact}$, $\ref{item:c_OF_otp_charact}$: The proof follows from considerations similar to those in the proof of items~$\ref{item:b_SF_otp_charact}$ and $\ref{item:c_SF_otp_charact}$ of Theorem~\ref{thm:SF_otp_charact} and is therefore omitted.
\qed

\textbf{Proof of Theorem~\ref{thm:DF_otp_charact}}
The proof follows from the duality of Proposition~\ref{prop:dual_net} and considerations similar to those in the proof of Theorem~\ref{thm:SF_otp_charact} and is therefore omitted.
\qed

\section{Proofs of Section~\ref{sec:feedback}}

\textbf{Proof of Theorem~\ref{thm:friend}}
To prove that $F=UZ^\top + F_q$, where $U$ is s.t. $AZ = Z X +BU$ holds for some $ X$, it is sufficient to consider equation~\eqref{eq:friend} and observe that when $ Z$ is a basis matrix for $ \Zm$, then $ Z^\top Z = I$, where $I$ denotes the identity matrix.
The formula for $F$ assumes that the matrix $U$ is given. We prove in the following that $U$ can be rewritten as a function of $A$, $B$ and $Z$. We first start by noting that $AZ = ZX+BU=[Z\quad B][X^\top\quad U^\top]^\top$. Therefore,
\beq\label{eq:pseud1}
[X^\top\quad U^\top]^\top=[Z\quad B]^\dagger AZ+\Gamma\alpha,
\eeq
where $[Z\quad B]^\dagger$ is the pseudo-inverse of $[Z\quad B]$, $\Gamma$ is a matrix whose columns form a basis for $\ker[Z\quad B]$ and $\alpha$ is a free parameter matrix of dimensions $\dim(\ker[Z\quad B])\times|\Zm|$, whose columns specify arbitrary linear combinations of the basis vectors of $\ker[Z\quad B]$. 
By construction, condition $\Vmin=\partial_+(\Zm^\circ(\Vmin),A)$ implies that there are no submatrices of $B$ whose range contains $\partial_+(\Zm^\circ(\Vmin),A)$. 
As a consequence, $\Zm^\circ(\Vmin)\cap\Vmin=\emptyset$ ensures that 
the columns of $[Z\quad B]$ are orthonormal, thus the pseudo-inverse becomes $[Z\quad B]^\dagger=[Z\quad B]^\top$. Hence, \eqref{eq:pseud1} simplifies to $[X^\top\quad U^\top]^\top=[Z\quad B]^\top AZ$, which leads to $U=B^\top AZ$, that inserted in equation~\eqref{eq:friend} gives $F=B^\top A+F_q$.
\qed

\textbf{Proof of Theorem~\ref{thm:output_inj}}
From Proposition~\ref{prop:dual_net} we know that $\Sm^\circ=\Sm^\circ(\Vmout,\,\Dm, A)$ is s.t. $\Sm^{\circ\perp}=\Zm^\circ(\Vmout,\Vm\smallsetminus\Dm,A^\top)$, thus $\partial_-(\Sm^\circ,A)=\partial_+(\Sm^{\circ\perp},A^\top)=\partial_+(\Zm^\circ(\Vmout,\Vm\smallsetminus\Dm,A^\top),A^\top)$. If we call $\hat{A}=A^\top$ and $\hat{B}=C^\top$, and $M=N^\top$ we can write $\hat{\Zm}^\circ(\Vmout)=\Zm^\circ(\Vmout,\Vm\smallsetminus\Dm,\hat{A})$. Therefore, by Theorem~\ref{thm:friend} we know that for the $(\hat{A},\hat{B})$-invariant set $\hat{\Zm}^\circ(\Vmout)$, if $\Vmout=\partial_+(\hat{\Zm}^\circ(\Vmout),\hat{A})$, then $F=\hat{B}^\top\hat{A}+F_q$ is the unique friend of $\hat{\Zm}^\circ(\Vmout)$, i.e., $(\hat{A}-\hat{B}(\hat{B}^\top\hat{A}+F_q))\hat{\Zm}^\circ(\Vmout)\subseteq\hat{\Zm}^\circ(\Vmout)$. By substituting, we get $(A^\top-C^\top(CA^\top+MS^\top))\Sm^{\circ\perp}\subseteq\Sm^{\circ\perp}\stackrel{\text{Lemma~\ref{lem:dual}}}{\iff}(A^\top-C^\top(CA^\top+MS^\top))^\top\Sm^\circ\subseteq\Sm^\circ$, i.e., $(A-HC)\Sm^\circ(\Vmout)\subseteq\Sm^\circ(\Vmout)$,
with $H=AC^\top+SN$.
\qed

\textbf{Proof of Theorem~\ref{thm:G}}
In general $\Sm^\circ(\Vmout)\subseteq\Wm\subseteq\Zm^\circ(\Vmin)$, i.e., $\Sm^\circ\neq\Zm^\circ$. The structure of the friend $G$ of $\Wm$, whose existence is guaranteed by Theorem~\ref{thm:ddpof}, indicates that when interpreted as controlled invariant set, the friend $F$ s.t. $(A-BF)\Wm\subseteq\Wm$ must have structure $F=GC$. Similarly, when $\Wm$ is interpreted as conditioned invariant set the friend $H$ s.t. $(A-HC)\Wm\subseteq\Wm$ must have structure $H=BG$. From Theorem~\ref{thm:friend}, condition $\partial_+(\Wm,A)=\Vmin$ guarantees that $B^\top A+F_q$ is a state-feedback friend of $\Wm$ and it is unique if $\Wm=\Zm^\circ$. In fact, $B^\top A$ may remove edges with head in $\Vmin$ but tail not in $\partial_-(\Wm,A)$. Similarly, from Theorem~\ref{thm:output_inj}, condition $\partial_-(\Wm,A)=\Vmout$ guarantees that $AC^\top+H_q$ is an output-injection friend of $\Wm$ and it is unique if $\Wm=\Sm^\circ$. In fact, $AC^\top$ may remove edges with tail in $\Vmout$ but head not in $\partial_+(\Wm,A)$. The output-feedback constraint forces the action of $F$ to be projected onto $\Vmout$, leading to $F=(B^\top A+F_q)C^\top C$. Condition $\partial_-(\Wm,A)=\Vmout$ ensures that the action of $F$ is to remove only the edges with tail in the in-boundary of $\Wm$. Similarly, the output-feedback constraint forces the action of $H$ to be projected onto $\Vmin$, leading to $H=BB^\top(AC^\top+H_q)$. Condition $\partial_+(\Wm,A)=\Vmin$ ensures that the action of $H$ is to remove only the edges with head in the out-boundary of $\Wm$. Under those assumptions it is clear that $F_q(C^\top C)=0$ and $(BB^\top)H_q=0$. The resulting output-feedback friend of $\Wm$ has therefore the unique form given by $G = B^\top A C^\top$.
\qed

\textbf{Proof of Theorem~\ref{thm:COM}}
Once $\Vmin$ and $\Vmout$ are assigned, we know from Theorems~\ref{thm:friend} and \ref{thm:output_inj} that the conditions $\Vmin=\partial_+(\Zm^\circ(\Vmin),A)$ and $\Vmout=\partial_-(\Sm^\circ(\Vmout),A)$ lead to the unique friends $F=B^\top A+F_q$ and $H=AC^\top+H_q$ of $\Zm^\circ$ and $\Sm^\circ$, respectively. 
Since $(\Sm^\circ(\Vmout),\Zm^\circ(\Vmin))$ is a $(C,A,B)$-pair, the only edges exiting $\Sm^\circ$ and landing outside of $\Zm^\circ$ must have head in $\partial_+(\Zm^\circ(\Vmin),A)$. 
A matrix $G$ s.t. $(A-BGC)\Sm^\circ\subseteq\Zm^\circ$ exists because of Lemma~\ref{lem:cab_SZ} (in Appendix~\ref{sec:compensator-design}). Since $C^\top C$ and $BB^\top$ are the orthogonal projections onto $\Vmout$ and $\Vmin$, respectively, and since $\Vmout=\partial_-(\Sm^\circ(\Vmout),A)$ and $\Vmin=\partial_+(\Zm^\circ(\Vmin),A)$, $G=B^\top AC^\top$ and it is unique, because of Theorem~\ref{thm:G}. Substituting in~\eqref{eq:comp} yields the full-order compensator. 
Finally, since $\Sm^\circ\subseteq\Zm^\circ$, $|\Zm^\circ\smallsetminus\Sm^\circ|=|\Zm^\circ|-|\Sm^\circ|=k$. Therefore, the reduced-order compensator of order $k$ is obtained by~\eqref{eq:red_comp} with the projection matrix $P\in\mathbb{R}^{k\times n}$ obtained by stacking the $k$ canonical row vectors corresponding to the nodes in $\Zm^\circ\smallsetminus\Sm^\circ$, and for which $PK_qP^\top=0$, $PH_q=0$ and $F_qP^\top=0$.
\qed

\section{Dynamical feedback case: Compensator synthesis and closed-loop system}
\label{sec:compensator-design}


If we consider the open-loop system~\eqref{eq:lin-syst2}, the existence of a dynamical feedback corresponds to the following closed-loop system (see \cite{trentelman2012control} for more details):
\beq
\begin{aligned}\label{eq:closed}
\begin{bmatrix}
\dot{x} \\
\dot{\hat{x}}
\end{bmatrix}
&=
\underbrace{\begin{bmatrix}
A - BGC & -BM \\
LC & K
\end{bmatrix}}_{=:A_c}
\begin{bmatrix}
x \\
\hat{x}
\end{bmatrix}+
\underbrace{\begin{bmatrix}
D \\
0
\end{bmatrix}}_{=:D_c} w \\
z &=
\begin{bmatrix}
T & 0
\end{bmatrix}
\begin{bmatrix}
x \\
\hat{x}
\end{bmatrix}=T_c\begin{bmatrix}
x \\
\hat{x}
\end{bmatrix}.
\end{aligned}
\eeq
The following result connects the existence of an $A_c$-invariant subspace for the closed-loop system to the existence of a $(C,A, B)$-pair.
Consider a subspace $\Wm_c$ of the extended state-space $\Xm\times \hat{\Xm}$, $\hat{\Xm}\subseteq\Xm=\mathbb{R}^n$, and let
$ \pi(\Wm_c):=\{x\in\Xm\,\,\text{s.t.}\,\,\exists\,\hat{x}\in\hat{\Xm}\,\,\text{s.t.}\,\,\begin{bmatrix}
x^\top &
\hat{x}^\top
\end{bmatrix}^\top\in\Wm_c\}$ and $\cap(\Wm_c):=\{x\in\Xm\,\,\text{s.t.}\,\,\begin{bmatrix}
x^\top &
0
\end{bmatrix}^\top\in\Wm_c\}$ be, respectively, the projection of $ \Wm_c$ onto $ \Xm$ and the intersection of $ \Wm_c$ with $ \Xm$.
\begin{theorem}[\textbf{Theorem~6.2 of \cite{trentelman2012control}}]
If a subspace $\Wm_c\subseteq\Xm\times \hat{\Xm}$ is $A_c$-invariant, with $A_c$ as in~\eqref{eq:closed}, then $(\cap(\Wm_c),\pi(\Wm_c))$ is a $(C,A, B)$-pair.
\label{thm:ActoCAB}
\end{theorem}

The synthesis of a decoupling compensator for system~\eqref{eq:closed} is derived in Theorem~\ref{thm:full} (see Theorem~6.4 of \cite{trentelman2012control} for a proof). We first use the following Lemma.
\begin{lemma}[\textbf{Lemma~6.3 of \cite{trentelman2012control}}]\label{lem:cab_SZ}
    If $(\Sm,\Zm)$ is a $(C,A,B)$-pair, then $\exists\,G\in\mathbb{R}^{m\times p}$ s.t. $(A-BGC)\Sm\subseteq\Zm$.
\end{lemma}

\begin{theorem}[\textbf{Full-order compensator}]\label{thm:full}
Let $(\Sm,\Zm)$ be a $(C,A,B)$-pair. Choose $G\in\mathbb{R}^{m\times p}$ s.t. $(A-BGC)\Sm\subseteq\Zm$, $F\in\mathbb{R}^{m\times n}$ s.t. $(A-BF)\Zm\subseteq\Zm$ and $H\in\mathbb{R}^{n\times p}$ s.t. $(A-HC)\Sm\subseteq\Sm$. The dynamical feedback
    \beq
    \begin{split}\label{eq:comp}
        \dot{\hat{x}}&=(A-BF-HC+BGC)\hat{x}+(H-BG)y\\&=K\hat{x}+Ly,\\
        u&=-(F-GC)\hat{x}-Gy=-M\hat{x}-Gy,
    \end{split}
    \eeq
    is s.t. $\Sm=\cap(\Wm_c)$ and $\Zm=\pi(\Wm_c)$ for the $A_c$-invariant subspace of $\Xm\times\Xm$ defined by $\Wm_c=\{\begin{bmatrix}
x_\Sm^\top + x_{\Zm}^\top & x_{\Zm}^\top
\end{bmatrix}^\top\,\,\text{s.t.}\,\,x_\Sm\in\Sm,\,x_\Zm\in\Zm\}$, with $A_c$ as in~\eqref{eq:closed}. The 4-tuple $(K,L,M,G)$ defined by~\eqref{eq:comp} is called full-order compensator.  
\end{theorem}

The following theorem specifies the dimension of the reduced-order observer needed to solve the problem.
\begin{theorem}[\textbf{Proposition~2.3 of \cite{schumacher1980compensator}}]\label{thm:red}
Let $(\Sm,\Zm)$ be a $(C,A,B)$-pair. Then there exists an extension-space $\hat{\Xm}$ of dimension $|\Zm| - |\Sm|$, and matrices $K,\,M,\,L$ and $G$, s.t. $\Sm=\cap(\Wm_c)$ and $\Zm=\pi(\Wm_c)$ for some $A_c$-invariant subspace $\Wm_c$ of $\Xm\times \hat{\Xm}$, and $A_c$ as in~\eqref{eq:closed}.
\end{theorem}

According to Theorem~\ref{thm:red}, it is possible to find an extension-space $\hat{\Xm}$ of dimension $|\Zm| - |\Sm|$. It is also possible to construct a reduced-order compensator by defining a projection map on the quotient space $\Zm/\Sm$. 
By the equivalence $e_i\sim v_i$, the projection is uniquely determined by the set $\Zm\smallsetminus\Sm$, as shown in the following Lemma, whose proof is omitted for lack of space.

\begin{lemma}\label{lem:proj}
Let \( \Xm = \mathbb{R}^n \), and let \( \Zm \subseteq \Xm \) be the subspace spanned by a subset of the canonical basis vectors, i.e., \( \Zm = \mathrm{span}\{e_{i_1}, \dots, e_{i_k}\} \). Let \( \Sm = \mathrm{span}\{e_{j_1}, \dots, e_{j_r}\} \subseteq \Zm \), with \( \{j_1, \dots, j_r\} \subseteq \{i_1, \dots, i_k\} \). Define the linear map \( P : \mathbb{R}^n \to \mathbb{R}^{k - r} \) by $P x = \begin{bmatrix} e_{m_1} & \dots & e_{m_{k-r}} \end{bmatrix}^\top x$,
where \( \{m_1, \dots, m_{k-r}\} = \{i_1, \dots, i_k\} \smallsetminus \{j_1, \dots, j_r\} \). Then \( P \) is a projection from \( \Zm \) onto the quotient space \( \Zm / \Sm \), with the following properties:
\begin{enumerate}[label=(\alph*)]
    \item \( P \) extracts the coordinates of \( x \in \mathbb{R}^n \) corresponding to the $k-r=|\Zm|-|\Sm|$ nodes in \( \Zm \smallsetminus \Sm \);
    \item \( \ker P|_{\Zm} = \Sm \) and \( \im P|_{\Zm} \cong \Zm / \Sm \cong \mathbb{R}^{k - r} \);\label{item:b_lem_proj}
    \item \( PP^\top = I_{k - r} \).
\end{enumerate}
\end{lemma}

\begin{theorem}[\textbf{Reduced-order compensator}]\label{thm:reduced}
Let $(\Sm,\Zm)$ be a $(C,A,B)$-pair. Consider the full-order compensator $(K,L,M,G)$ defined by~\eqref{eq:comp} and a projection $P$ from \( \Zm \) onto the quotient space \( \Zm/\Sm \) as in Lemma~\ref{lem:proj}. The reduced dynamical feedback
    \beq\label{eq:red_comp}
    \begin{split}
        \dot{\hat{x}}_r&=PKP^\top\hat{x}_r+PLy,\\
        u&=-MP^\top\hat{x}-Gy,
    \end{split}
    \eeq
    is s.t. $\Sm=\cap(\Wm_c)$ and $\Zm=\pi(\Wm_c)$ for the $A_c$-invariant subspace of $\Xm\times\Zm/\Sm$ defined by $\Wm_c=\{\begin{bmatrix}
x^\top
&
x^\top P^\top
\end{bmatrix}^\top\,\,\text{s.t.}\,\,x\in\Zm\}$, with $A_c$ equal to
   \begin{equation*}
\begin{bmatrix}
A - BGC & -B(F-GC)P^\top \\
P(H-BG)C & P(A-BF-HC+BGC)P^\top
\end{bmatrix}.
   \end{equation*} The 4-tuple $(PKP^\top,PL,MP^\top,G)$ is called reduced-order compensator of order $|\Zm|-|\Sm|$.  
\end{theorem}
\begin{proof}
See Proposition~2.3 of \cite{schumacher1980compensator}.
\end{proof}

\bibliographystyle{plain}

\end{document}